\documentclass[reqno]{amsart}
\usepackage{amssymb,amsmath,amsfonts,amscd,amsthm,pb-diagram,mathrsfs,amsmath,color}
\usepackage{amsbsy,bm}

\newtheorem{theorem}{Theorem}[section]

\newtheorem{definition}{Definiton}[section]

\newtheorem{lemma}{Lemma}[section]
\newtheorem{proposition}{Proposition}[section]
\newtheorem{corollary}{Corollary}[section]
\theoremstyle{remark}
\newtheorem{remark}{Remark}[section]

\def\O{\Omega}

\def\n{\nabla}

\def\p{\partial}

\def\n{\nabla}

\def\o{\omega}
\def\O{\Omega}
\def\p{\partial}

\def\g{\gamma}

\def\k{\kappa}

\def\s{\sigma}

\def\n{\nabla}
\def\<{\langle}
\def\>{\rangle}

\def\n{\nabla}

\def\o{\omega}
\def\O{\Omega}
\def\p{\partial}

\def\g{\gamma}

\def\s{\sigma}

\def\Rn{\mathbb R^n}

\def\R{\mathbb{R}}

\def\rr{\mathbb{R}}

\begin{document}
\title[Anisotropic $p$-capacity and anisotropic Minkowski inequality]
{Anisotropic $p$-capacity and anisotropic Minkowski inequality}

\author{Chao Xia, Jiabin Yin}%----------Author 1,2
\address{School of Mathematical Sciences, Xiamen University, Xiamen, 361005, China}
\email{chaoxia@xmu.edu.cn}
\address{School of Mathematical Sciences, Xiamen University, Xiamen, 361005, China}
\email{jiabinyin@126.com}

\thanks{CX was  supported by  NSFC (Grant No. 11871406).
}

%\thanks{The authors were supported by NSF of China, Grant Numbers .}

\date{}


\begin{abstract}
In this paper, we prove a sharp anisotropic $L^p$ Minkowski inequality involving the total
$L^p$ anisotropic mean curvature and the anisotropic $p$-capacity, for any bounded domains with smooth boundary in $\rr^n$.
As consequences, we obtain an anisotropic Willmore inequality,  a sharp anisotropic Minkowski inequality for outward $F$-minimising sets and a sharp volumetric anisotropic Minkowski inequality.  For the proof, we utilize a nonlinear potential theoretic approach which has been recently developed in \cite{AFM1}. \end{abstract}

\maketitle

\medskip

\tableofcontents

%==========================================================================
\numberwithin{equation}{section}
\section{Introduction}\label{sect:1}

The Minkowski inequality in the theory of convex geometry says that if $\O\subset\mathbb{R}^n$ is a convex domain with smooth boundary, then
\begin{eqnarray}\label{mink-1}
\frac{1}{\omega_{n-1}}\int_{\p \O} \frac{H}{n-1} d\sigma \ge\left(\frac{|\p \O|}{\o_{n-1}}\right)^\frac{n-2}{n-1},
\end{eqnarray}
 with equality holding if and only if $\O$ is a round ball. Here $H$ is the mean curvature of $\p\O$ with respect to the outward unit normal and $\o_{n-1}=|\mathbb{S}^{n-1}|,$ the area of unit $(n-1)$-sphere.
%Minkowski inequality can be reformulated to
%$$V_{(2)}(\O, B_1)\ge V_{1}(\O, B_1)^{}{\rm Vol}(B_1).$$

%The Minkowski inequality is a direct consequence of Alexandrov-Fenchel inequalities for mixed volumes.

In recent decades, many mathematicians are interested in the question whether \eqref{mink-1} holds for non-convex domains.
By using smooth solutions to inverse mean curvature flow, Guan-Li \cite{GL} proved that \eqref{mink-1} holds for star-shaped domains with mean convex boundary. Huisken-Ilmanen \cite{HI} developed a level set weak formulation for inverse mean curvature flow, which was adopted by Huisken \cite{Hui} to prove \eqref{mink-1} for outward minimising sets (see also \cite{FS}).
Quite recently, a powerful and elegant approach based on nonlinear potential theory has been developed by Agostiniani-Fogagnolo-Mazzieri in \cite{AFM1} (see also related papers \cite{AFM, AM, FMP}), which enables them to establish a sharp $L^p$ version of Minkowski inequality involving the $p$-capacity of $\O$ ($1<p<n$)
for general bounded domain with smooth boundary
\begin{equation}\label{eqn:0.1}
\frac1{\omega_{n-1}}
\int\limits_{\partial\Omega}\Big|\frac{H}{n-1}\Big|^p d\sigma\ge \left[\big(\frac{p-1}{n-p}\big)^{p-1}\frac1{\o_{n-1}}{\rm Cap}_{p}
(\Omega)\right]^{\frac{n-p-1}{n-p}}.
\end{equation}
Here ${\rm Cap}_{p}
(\Omega)$ is called by the $p$-capacity of $\O$, given by
$${\rm Cap}_{p}
(\Omega)=\inf\Big\{\int_{\rr^n}|\nabla u|^p dx: u\in C_c^\infty(\rr^n), u\ge 1\hbox{ on }\O\Big\}.$$
We remark that various geometric inequalities involving the mean curvature integral, the capacity and the surface area have been considered by Xiao \cite{Xiao2}. In particular, \eqref{eqn:0.1} has been proved in \cite[Theorem 3.1]{Xiao2} in the case of star-shaped, mean convex domains.
As $p\to 1$ in \eqref{eqn:0.1}, Agostiniani-Fogagnolo-Mazzieri have showed that the limit version of \eqref{eqn:0.1} recovers \eqref{mink-1} for outward minimising sets by Huisken \cite{Hui}. Moreover, they are able to obtain a volumetric version of Minkowski inequality
\begin{equation}\label{vol-mink}
\frac1{\o_{n-1}}\int\limits_{\partial\Omega}\Big|\frac{H}{n-1}\Big|
d\sigma \ge \left(\frac{|\O|}{|\mathbb{B}^n|}\right)^{\frac{n-2}{n}}.
\end{equation}
for any bounded domains. Previously, \eqref{vol-mink} has been shown by Trudinger \cite{Tr} and Chang-Wang \cite{CW} (see also \cite{Qiu}) for bounded domains with $2$-convex boundary.
 The idea of Agostiniani-Fogagonolo-Mazzieri in \cite{AFM1}  is based on  a conformal change of the Euclidean
metric and a monotone integral over level sets of $p$-capacitary potentials under the conformal metric. Actually, it turns out that the results can be obtained by working directly on the Euclidean
metric rather than the conformal metric. It is worth highlighting that their approach can be used  to obtain sharp geometric inequalities for hypersurfaces in manifolds with nonnegative Ricci curvature \cite{AFM,AM,FMP} as well as for static manifolds \cite{AM1, BM, BM1}, see also \cite{BrayM, Xiao1}.

%In this paper, we are interested in a Minkowski inequality for mixed volumes. For two convex domains $\O$ and $\mathcal{W}$ in $\rr^n$,  The mixed volumes $V_{(k)}(\O, \mathcal{W})$ are defined by
%$$\hbox{Vol}\left(\O+t\mathcal{W}\right)= \sum_{k=0}^{n}\binom{n}{k}t^kV_{(k)}(\O, \mathcal{W}).$$
%When $\mathcal{W}$ has smooth and strictly convex boundary and $\O$ has smooth boundary, $V_{(k)}(\O, \mathcal{W})$ can be interpreted as anisotropic curvature integrals:
%\begin{eqnarray}\label{mixed-vol}
%V_{(k)}(\O,\mathcal{W})=\frac{1}{n\binom{n-1}{k-1}}\int_{\p \O}\sigma_{k-1}(\k^F)  F(\nu)d\mu_{\p \O}, \quad k=1,\cdots,n-1.
%\end{eqnarray}
%where $\sigma_k(\k^F)$ is the $k$-th elementary symmetric function on the anisotropic principal curvature $\k^F$.
%$V_{(k)}(\O, \mathcal{W})$ satisfies a class of Alexandrov-Fenchel inequalities. As a special case, the following Minkowski inequality for mixed volumes holds:
%\begin{eqnarray}\label{eqn:0.2}
%V_{(2)}(\O, \mathcal{W})\ge V_{1}(\O, \mathcal{W})^{}{\rm Vol}(\mathcal{W}).
%\end{eqnarray}

%When $\mathcal{W}=\mathbb{B}$ is the unit ball, \eqref{eqn:0.2} reduces to \eqref{mink-1}.

%The Minkowski inequality \eqref{mink-1} is  a special case of
%It is worth recalling that \eqref{eqn:0.2} is a special one in .
%Here

 In this paper, we are concerned with an anisotropic version of Minkowski inequality.
Let $F\in C^{\infty}(\R^{n}\setminus\{0\})$ be a Minkowski norm on $\R^n$. The unit Wulff ball $\mathcal{W}_F$ with respect to $F$ (centered at the origin) is given by $$\mathcal{W}_F=\{F^o(x)<1\},$$ where $F^o$ is the dual norm of $F$.
For a bounded domain $\O$ with smooth boundary, the {\it anisotropic area} of $\p \O$ is defined by $$|\p\O|_F=\int_{\p\O}F(\nu)d\sigma.$$
The well-known Wulff theorem  (see e.g  Theorem 20.8 in \cite{Maggi}) says that Wulff balls are the only minimizers for the {\it anisotropic isoperimetric problem}. Equivalently,
the {\it Wulff inequality} holds true:
\begin{equation}\label{wulff-ineq}
  |\p \Omega|_F\geq n|\mathcal{W}_F|^{\frac1n}|\Omega|^{1-\frac1n}.
\end{equation}
Equality in \eqref{wulff-ineq} holds if and only if $\Omega$ is a Wulff ball.

The anisotropic mean curvature $H_F$ of $\p\O$ arises from the first variational formula for $|\p\O|_F$.
The anisotropic Minkowski inequality
says that  if $\O\subset\mathbb{R}^n$ is a {\it convex} domain with smooth boundary, then
\begin{eqnarray}\label{mink-2}
\frac{1}{\kappa_{n-1}}\int_{\p \O} \frac{H_F}{n-1} F(\nu)d\sigma\ge\left(\frac{|\p\O|_F}{\k_{n-1}}\right)^\frac{n-2}{n-1},
\end{eqnarray}
 with equality holding if and only if $\O$ is a Wulff ball. Here $\k_{n-1}=n|\mathcal{W}_F|$. When $F$ is the Euclidean norm $F(\xi)=|\xi|$, then \eqref{wulff-ineq} and \eqref{mink-2} reduce to the classical isoperimetric inequality and the Minkowski inequality \eqref{mink-1} respectively.
Inequality \eqref{mink-2} is in fact equivalent to the following Minkowski inequality for mixed volumes
\begin{eqnarray}\label{eqn:0.2'}
V_{(2)}(\O, \mathcal{W}_F)\ge V_{1}(\O, \mathcal{W}_F)^{\frac{n-2}{n-1}}|\mathcal{W}_F|^{\frac1{n-1}}
\end{eqnarray}
where $V_{i}(\O, \mathcal{W}_F)$ is the mixed volume for two convex bodies  $\bar{\O}$ and $\bar{\mathcal{W}_F}$, which is given by
$$\left|\O+t\mathcal{W}_F\right|= \sum_{k=0}^{n}\binom{n}{k}t^kV_{(k)}(\O, \mathcal{W}_F).$$
Inequality \eqref{eqn:0.2'} is a special one for a class of Alexandrov-Fenchel inequalities for mixed volumes in the theory of convex bodies, see e.g. Schneider's celebrated book \cite{Sch} Section 7.3, or the Introduction in \cite{Xia}.

Regarding \eqref{mink-2}, the first named author \cite{Xia} showed that \eqref{mink-2} holds for star-shaped domains with $F$-mean convex ($H_F\ge 0$) boundary, by using the smooth solution to the inverse anisotropic mean curvature flow, in the same spirit of Guan-Li \cite{GL}.
Della Pietra, Gavitone and the first named author \cite{DGX} has considered the level set formulation of the inverse anisotropic mean curvature flow, in the spirit of Huisken-Ilmanen \cite{HI}, and proved the existence of the weak solutions by using $L^p$-type approximation. However, such approximation seems not sufficient to prove \eqref{mink-2} for  outward $F$-minimising sets.

In this paper, we will use the method developed in  \cite{AFM1, AM, FMP} to study \eqref{mink-2} for non-convex domains. %under outward $F$-minimizing condition.
%\begin{definition}\label{Outward $F$-minimizing}
%\end{definition}
 For $1<p<n$, the {\it anisotropic $p$-capacity} of $\O$ is defined by
$${\rm Cap}_{F, p}(\Omega)=\inf\Big\{\int_{\mathbb R^n}F^p(\nabla u) dx: u\in C_c^\infty(\rr^n), u\ge 1\hbox{ on }\O\Big\}.$$
We remark that there are early consideration for basic properties of the anisotropic $p$-capacity, see e.g. Maz'ya's book \cite{Maz}, Section 2.2. Also, the Brunn-Minkowski theory has been extended to the anisotropic $p$-capacity \cite{AGHLV, ALSV}.

In the following we generalize \eqref{eqn:0.1} to the anisotropic case.
\begin{theorem}%[$L^p$-Anisotropic Minkowski Inequality]
\label{thm:1.1}
Let {$\Omega\subset\mathbb R^n(n\geq3)$} be a bounded domain with smooth
boundary. Then, for every $1 < p < n$, the following inequality holds
\begin{equation}\label{eqn:1.02}
\frac1{\kappa_{n-1}}
\int\limits_{\partial\Omega}\Big|\frac{H_F}{n-1}\Big|^pF(\nu)d\sigma\ge \left[\big(\frac{p-1}{n-p}\big)^{p-1}\frac1{\kappa_{n-1}}{\rm Cap}_{F,p}
(\Omega)\right]^{\frac{n-p-1}{n-p}}.
\end{equation}
%where $C_p^F(\Omega)$  is the normalised anisotropic $p$-capacity of $\Omega$ introduced in Definition \ref{def:2.1}.
Moreover, equality holds in \eqref{eqn:1.02} if and only if $\Omega$ is a Wulff  ball.
\end{theorem}
As a direct consequence (with $p=n-1$), we get the {\it anisotropic Willmore inequality}.
\begin{theorem}%[$L^p$-Anisotropic Minkowski Inequality]
\label{willmore}
Let {$\Omega\subset\mathbb R^n(n\geq3)$} be a bounded domain with smooth
boundary. Then
\begin{equation}\label{willmore'}
\int\limits_{\partial\Omega}\Big|\frac {H_F}{n-1}\Big|^{n-1}F(\nu)d\sigma\ge \kappa_{n-1}.
\end{equation}
%where $C_p^F(\Omega)$  is the normalised anisotropic $p$-capacity of $\Omega$ introduced in Definition \ref{def:2.1}.
Moreover, equality holds in \eqref{willmore'} if and only if $\Omega$ is a Wulff  ball.
\end{theorem}
Theorem \ref{thm:1.1} are a special case of  Theorem \ref{thm-pq}. %We shall use the method developed in \cite{AFM1} based on nonlinear potential theory to prove Theorem \ref{thm-pq}.
Consider the anisotropic $p$-capacitary potential $u$, which satisfies
\begin{equation*}
\left\{
\begin{aligned}
\Delta_{F, p}u=&0\ \ {\rm in}\ \ \mathbb R^n\setminus\bar{\Omega}\\
u=&1\ \ {\rm on} \ \ \partial\Omega\\
u(x)\rightarrow& 0\ \ {\rm as}\ \ |x|\rightarrow\infty,
\end{aligned}\right.
\end{equation*}
where  $\Delta_{F, p}$ is the anisotropic $p$-Laplace operator (see Section \ref{sect:2}). It is not hard to see that
$${\rm Cap}_{F, p}(\Omega)=\int_{\mathbb R^n\setminus\bar{\Omega}}F^p(\n u) dx= \int_{\partial{\Omega}}F^{p-1}(\n u)F(\nu)d\sigma.$$
Let $1<p<n$ and $q\ge 1+\frac{1}{p^*}$, where $p^*:=\frac{(n-1)(p-1)}{n-p}$. Let $\Phi_{p,q}: [1,\infty)\to \rr$ be defined by
\begin{equation*}
\Phi_{p,q}(\tau):=\tau^{(q-1)p^*}\int\limits_{\{u=1/\tau\}}F^{q(p-1)}(\nabla u)F(\nu)d\sigma.
\end{equation*}
The asymptotic behavior of $\Phi_{p,q}$ corresponds to ${\rm Cap}_{F, p}(\Omega)$ and the key ingredient to prove Theorem \ref{thm-pq} is that $\Phi_{p,q}$ is an (essentially) monotone non-increasing function. See Sections \ref{sect:4} and \ref{sect:5}.

Next by letting $p\to 1$ in \eqref{eqn:1.02}, we get the following extended anisotropic Minkowski Inequality.
\begin{theorem}[]\label{thm:1.2}
If {$\Omega\subset\mathbb R^n(n\geq3)$} is a bounded domain with smooth boundary, then
\begin{equation}\label{eqn:1.03}
\frac1{\k_{n-1}}\int\limits_{\partial\Omega}\Big|\frac{H_F}{n-1}\Big|F(\nu)
d\sigma \ge \left(\frac1{\k_{n-1}}\inf_{\substack{\O\subset U\\ \p U \hbox{smooth}}}|\p U|_F\right)^{\frac{n-2}{n-1}}.
\end{equation}
%where $\Omega^*$ is the strictly outward $F$-minimising hull of $\Omega$ defined as in \eqref{eqn:6.2}.
%Moreover, the dimensional constants appearing here are optimal, in the sense that
%
%\begin{equation*}
%{\rm min}\Big\{\left(|\p \Omega^*|_F\right)^{-\frac{n-2}{n-1}}\int\limits_{\partial\Omega}|H_F|d\sigma\,\Big|\, \Omega\Subset\mathbb R^n,\,
%{\rm wiht}\ \partial\Omega\ {\rm smooth}\Big\}=(n-1)|\mathcal{W}_F|^{\frac1{n-1}}
%\end{equation*}
%and
Moreover, equality in \eqref{eqn:1.03} is achieved by Wulff balls.
\end{theorem}

As consequences, we are able to deduce a sharp anisotropic Minkowski inequality for outward $F$-minimising sets (see Definition \ref{ofm}) and a sharp volumetric anisotropic Minkowski inequality.
\begin{corollary}%[Minkowski Inequality for Outward Anisotropic Minimising Sets]
\label{cor:1.3}
If {$\Omega\subset\mathbb R^n(n\geq3)$} is a bounded outward $F$-minimising set with smooth boundary, then
\begin{equation}\label{eqn:1.04}
\frac1{\k_{n-1}}\int\limits_{\partial\Omega}\frac{H_F}{n-1} F(\nu)
d\sigma \ge \left(\frac{|\p \O|_F}{\k_{n-1}}\right)^{\frac{n-2}{n-1}}.
\end{equation}
Moreover, equality in \eqref{eqn:1.04} is achieved by Wulff balls. %the dimensional constants appearing here are optimal, in the sense that
%
%\begin{equation*}
%{\rm min}\Big\{\left(|\p \Omega^*|_F\right)^{-\frac{n-2}{n-1}}\int\limits_{\partial\Omega}|H_F|d\sigma\,\Big|\, \Omega\Subset\mathbb R^n,\,
%{\rm wiht}\ \partial\Omega\ {\rm smooth}\Big\}=(n-1)\kappa_{n-1}^{\frac1{n-1}}
%\end{equation*}
%and the minimum is achieved on Wullf balls.
Conversely, if the equality holds in \eqref{eqn:1.04} for some
bounded strictly outward $F$-minimising open set with smooth and strictly $F$-mean-convex boundary,
then $\Omega$ is a  Wulff ball.
\end{corollary}

\begin{corollary}%[Volumetric anisotropic Minkowski inequality]
\label{cor:1.4}
Let {$\Omega\subset\mathbb R^n(n\geq3)$} be a bounded set with smooth boundary. Then
\begin{equation}\label{eqn:1.05}
\frac1{\k_{n-1}}\int\limits_{\partial\Omega}\Big|\frac{H_F}{n-1}\Big|F(\nu)
d\sigma \ge \left(\frac{n|\O|}{\k_{n-1}}\right)^{\frac{n-2}{n}}.
\end{equation}
Moreover, equality holds  in \eqref{eqn:1.05} if and only if $\O$ is a Wulff ball. \end{corollary}

To end the introduction, we make some remark on the {\it strictly outward $F$-minimising hull}. Let $\O$ be a bounded domain with smooth boundary.
It is well-known that  the infimum
 $$\inf\limits_{\substack{\O\subset U\\ \p U \hbox{smooth}}}|\p U|_F$$ may not be achieved in general. In the setting of set of finite perimeter, the infimum
\begin{eqnarray}\label{inf-sfp}
\inf_{\substack{\O\subset E \\ \ E \hbox{ set of perimeter}}}P_F(E)
\end{eqnarray}
can be achieved, by the direct method in Calculus of Variations. Here
$$P_F(E)=\int_{\p^*E} F(\nu_E)d\mathcal{H}^{n-1}$$ is the {\it anisotropic perimeter} of $E$,
where $\p^*E$ is the reduced boundary of $E$ and $\nu_E$ is the outward measure theoretical unit normal to $E$, see e.g. \cite[Chapter 20]{Maggi}.
There is a so-called strictly outward $F$-minimising hull corresponding to $\O$ which attains the infimum, see  \cite{AFM1,FM,HI}. Similarly as in \cite[Definition 2.12]{FM},  the strictly outward $F$-minimising hull of $\Omega$ denoted by $\Omega^*$, can be reformulated as
the measure theoretic interior Int$(E)$ of any set $E$ that solves the minimisation problem
\begin{equation*}\label{eqn:6.2}
|E|=\inf\limits_{G\in {\rm SOFMBE(\Omega)}}|G|
\end{equation*}
where
\begin{equation*}
  {\rm SOFMBE}(\Omega)=\{G\,|\, \Omega\subseteq G\ \hbox{ and }\ G \hbox{ is bounded and strictly outward} \ F-{\rm minimising }\}
\end{equation*}
is the Strictly Outward $F$-Minimizing Bounded Envelopes of $\O$.
It turns out that $\O^*$ is the unique  maximal volume solution to the least anisotropic perimeter problem with obstacle $\O$, up to zero-measure modification, see \cite[Theorem 2.16]{FM} (In our case, the assumption on the existence of an exhaustion of $\rr^n$ by a sequence of strictly outward $F$-minimising sets is automatically satisfied. In fact, such exhaustion is given by the Wulff balls). In particular, $\O^*$ attains the infimum in \eqref{inf-sfp}, namely,
%, namely
%
%\begin{equation*}
%  {\rm Int}(E)=\Big\{x\in\mathbb R^n\Big|\, \lim\limits_{r\rightarrow0^+}\frac{|E\cap B(x,r)|}{|B(x,r)|}=1\Big\}.
%\end{equation*}
%As in \cite{HI}, $\Omega^*$ achieves the infimum in \eqref{inf-sfp}, namely
\begin{equation*}\label{eqn:6.3}
P_F(\Omega^*)=\inf_{\substack{\O\subset E \\ \ E \hbox{  set of perimeter}}}P_F(E).
\end{equation*}
It is direct to see that
\begin{equation*}\label{eqn:6.3'}
\inf_{\substack{\O\subset U\\ \p U \hbox{smooth}}}|\p U|_F\ge \inf_{\substack{\O\subset E \\ \ E \hbox{  set of perimeter}}}P_F(E)=P_F(\Omega^*).
\end{equation*}
Thus it follows from \eqref{eqn:1.03} that it holds
\begin{equation}\label{mink-last}
\frac1{\k_{n-1}}\int\limits_{\partial\Omega}\Big|\frac{H_F}{n-1}\Big|F(\nu)
d\sigma \ge \left(\frac{P_F(\O^*)}{\k_{n-1}}\right)^{\frac{n-2}{n-1}}.
\end{equation}
\begin{remark} In \cite{AFM1}, the isotropic version of \eqref{mink-last} has been shown. In their argument, a regularity result in \cite{SZW} and an exterior smooth approximation in \cite{Schm} for the strictly outward minimising hull $\O^*$ have to play a role. In this paper, by showing a stronger inequality \eqref{eqn:1.03}, we can avoid of using the regularity and an exterior smooth approximation for the anisotropic case, which seems unavailable in the reference.
\end{remark}

\medskip

The rest of the paper is organized as follows. In Section \ref{sect:2} we review the basic concepts in anisotropic setting, including  the anisotropic mean curvature for level sets and the anisotropic $p$-laplace operator. In particular, we prove a Kato-type identity for anisotropic $p$-harmonic functions.
In Section \ref{sect:3} we make a systematic study of anisotropic $p$-capacitary potential, focusing on its asymptotic behavior. In Sections \ref{sect:4} and \ref{sect:5},  we establish the essential and effective monotonicity properties for the crucial function $\Phi_{p, q}(\tau)$, and in turn, prove the $L^p$ anisotropic Minkowski inequality, Theorem \ref{thm:1.1}. In Section \ref{sect:6}, we study the outward $F$-minimising sets, the limit of ${\rm Cap}_{F, p}$ as $p\to 1^+$ and finally prove Theorem \ref{thm:1.2} and Corollaries \ref{cor:1.3} and \ref{cor:1.4}.

\

%%%%%%%%%%%%%%%%%%%%%%%%%%%%%%%%%%%%%%%%%%%%%%%%%%%%%%%%%%%%%%%%%%%%%%
\section{Preliminaries}\label{sect:2}
\subsection{Minkowski norm,  Wulff shape, anisotropic area}\

Let $F\in C^{\infty}(\R^{n}\setminus\{0\})$ be a {\it Minkowski norm} on $\R^n$, in the sense that
\begin{itemize}
\item[(i)] $F$ is a norm in $\mathbb{R}^{n}$, i.e.,  $F$ is a convex, even, $1$-homogeneous function satisfying $F(x)>0$ when $x\neq 0$ and $F(0)=0$;
\item[(ii)] $F$ satisfies a uniformly elliptic condition: $\n^2 (\frac12 F^2)$ is positive definite in $\mathbb{R}^{n}\setminus \{0\}$.
\end{itemize}

The {\it dual norm} $F^o\colon \Rn \rightarrow [0,+\infty[$
of $F$ is defined as
\begin{equation*}
F^o(x)=\sup_{\xi\ne 0} \frac{\langle \xi, x\rangle}{F(\xi)}.
\end{equation*}
$F^o$ is also a Minkowski norm on $\R^{n}$.

Furthermore,
\begin{equation*}
F(\xi)=\sup_{x\ne 0} \frac{\langle \xi, x\rangle}{F^o(x)}.
\end{equation*}
We remark that, throughout this paper,  we use conventionally $\xi$ as the variable for $F$ and $x$ as the variable for $F^o$.

Denote
\[
\mathcal W_F= \{  x\in  \Rn \colon F^o(x)< 1 \}.
\]
For the simplicity of notations, we will denote by $\mathcal W_F=\mathcal W$.
We call $\mathcal W$ the {\it unit Wulff ball} centered at the origin, and $\p \mathcal W$ the {\it Wulff shape}.

% We put
%$\kappa_n=|\mathcal W|$, where $|\mathcal W|$ denotes the Lebesgue measure
%of $\mathcal W$.
More generally, we denote $$\mathcal W_r(x_0)=r\mathcal W+x_0,$$ and call it the {\it Wulff ball of radius $r$ centered at $x_0$}. We simply denote $\mathcal W_r=\mathcal W_r(0)$.

The following properties of $F$ and $F^o$ hold true and will be frequently used in this paper (see e.g. \cite{DGX,WX4}).
\begin{proposition}\label{prop:2.1}
Let $F:\mathbb R^n\rightarrow[0,\infty)$ be a Minkowski norm. Then for any  $x, \xi \in \Rn\setminus \{0\}$, the
following holds:
\begin{enumerate}
\item  $\langle F_{\xi} (\xi) , \xi \rangle= F(\xi), \quad  \langle  F^{o}_x (x), x \rangle
= F^{o}(x).$
\item $ \sum_jF_{\xi_i\xi_j}(\xi)\xi_j=0$ for any $i=1,\ldots,n$.
\item $F(F_{x}^o(x) )=F^o( F_{\xi}(\xi))=1.$
\item $F^o(x)  F_{\xi}( F_x^o(x) ) =x, \quad F(\xi)
 F_{x}^o\left( F_{\xi}(\xi) \right) = \xi.$
\end{enumerate}
\end{proposition}

Let $\Omega\subset\mathbb R^n$ be a bounded open set with smooth boundary $\p\O$ and $\nu$ be its unit outward normal of $\p\O$.  The {\it anisotropic area} $|\p \Omega|_F$ of $\O$ is defined by
\begin{equation}\label{eqn:2.002}
 |\p\Omega|_F=\int_{\partial\Omega}F(\nu)d\sigma.
\end{equation}

Note that when $\O=\mathcal{W}$, the unit Wulff ball, one can check by the divergence theorem that
\begin{equation}\label{eqn:2.2}
 |\p \mathcal{W}|_F=\int_{\p \mathcal{W}} \frac{1}{|\n F^o|}d\sigma=\int_{\mathcal{W}}{\rm div }(x) dx=n|\mathcal{W}|.
\end{equation}
For notation simplicity, we denote $$\kappa_{n-1}= |\p \mathcal{W}|_F=n|\mathcal{W}|.$$

\medskip

\subsection{Anisotropic $p$-Laplacian}\

Let $u$ be twice continuous differentiable at $x\in \rr^n$.
We denote by $F_i, F_{ij}, \ldots$  the partial derivatives of $F$ and by $u_i, u_{ij},\ldots$ the partial derivatives of $u$,
$$F_i=\frac{\p F}{\p \xi_i},  \ F_{ij}=\frac{\p^2 F}{\p \xi_i\p \xi_j}, \ u_i=\frac{\p u}{\p x_i}, \ u_{ij}=\frac{\p^2 u}{\p x_i\p x_j}.$$

For $x$ such that $\n u(x)\neq 0$, denote
\begin{equation}\label{eqn:2.1}
 \begin{aligned}
&& a_{ij}(\nabla u)(x):=&\frac{\partial^2}{\partial\xi_i\partial\xi_j}(\frac12F^2)(\nabla u(x))=(F_iF_j+FF_{ij})(\nabla u(x)),\\
&& a_{ij,p}(\nabla u)(x):=&\frac{\partial^2}{\partial\xi_i\partial\xi_j}(\frac1pF^p)(\nabla u(x))=F^{p-2}(a_{ij}+(p-2)F_iF_j)(\nabla u(x)).
 \end{aligned}
\end{equation}

The {\it anisotropic Laplacian and $p$-Laplacian} of $u$ at regular points ($|\nabla u|\neq 0$) is given by
\begin{equation}\label{eqn:2.001}
 \begin{aligned}
 && \Delta_F u:=&a_{ij}(\nabla u)u_{ij},\\
&& \Delta_{F,p} u:=&a_{ij,p}(\nabla u)u_{ij}=F^{p-2}(\Delta_F u+ (p-2)F_iF_ju_{ij}),
 \end{aligned}
\end{equation}
See e.g. \cite{CS, WX1}.
For notation simplicity, we also introduce $$V(\xi)=\frac{1}{p}F^p(\xi).$$

and \begin{eqnarray}\label{W-def}
W_{ij}=\p_{x_j}\left(\p_{\xi_i}V(\nabla u)\right)=a_{ik,p}(\n u)u_{kj}.
\end{eqnarray}

One sees easily that $${\rm tr}(W)=\Delta_{F,p} u.$$

Let $1 \leq k \leq n$, be an integer. For a $n$-vector $\lambda=(\lambda_1,\ldots,\lambda_n)$, the $k$-th elementary symmetric function on $\lambda$ is defined by
\begin{equation*}
  \sigma_k(\lambda)=\sum\limits_{1\leq i_1<\cdots<i_k\leq n}\lambda_{i_1}\cdots\lambda_{i_k}.
\end{equation*}
Given a real matrix $B=(b_{ij})\in\mathbb R^{n\times n}$, we denote by $S_k(B)$ the sum of all
the principal minors of $B$ of order $k$. In particular, $S_1(B)={\rm tr}(B)$, the trace of $B$, and
$S_n(B)={\rm det}(B)$, the determinant of $B$.
When the eigenvalues $\lambda_B$ of $B$ are all real, it is clear that
$$S_k(B)=\sigma_k(\lambda_B).$$ For more properties on $S_k(B)$, we refer to \cite{DGX1,Reilly}.

For our purpose, we consider $k=2$. By setting
\begin{equation*}
  S^{ij}_2(B)={\rm tr}B\, \delta_{ij}- b_{ji},
\end{equation*}
then
\begin{equation}\label{eqn:2.8}
S_2(B)=\frac12S^{ij}_2(B)b_{ij}=\frac12\big(({\rm tr}B)^2-{\rm tr}B^2\big).
\end{equation}
Of particular interest for our purpose is $S_2(W)$, where $W$ is given by \eqref{W-def}.

Since ${\rm tr}W=\Delta_{F, p}u$, it holds
\begin{equation}\label{eqn:2.9}
S^{ij}_2(W)=\Delta_{F, p}u\, \delta_{ij}- a_{jk,p}(\n u)u_{ki}.
\end{equation}
Moreover, in this setting $S^{ij}_2(W)$ is divergence free, see e.g.\cite{BC, CS, DGX1}
\begin{equation}\label{eqn:2.10}
\p_{x_j}S^{ij}_2(W)=0.
\end{equation}
It follows from \eqref{eqn:2.8} and \eqref{eqn:2.10} that if  $\Delta_{F, p}u=0$, then
\begin{equation}\label{eqn:2.11}
\p_{x_j}(S^{ij}_2(W)V_{\xi_i}(\nabla u))= 2S_2(W)=-a_{ik,p}a_{jl,p}u_{kj}u_{il}.
\end{equation}
%where $V_{\xi_i}(\nabla u)=F^{p-1}(\nabla u)F_i$.

\medskip

\subsection{Anisotropic curvature for level sets}\

We recall the concept of anisotropic curvature for a hypersurface in $\mathbb R^n$. See e.g. \cite{WX1, Xia}.

Let $M$ be a smooth embedded hypersurface in $\mathbb R^n$ and $\nu$ be one unit
normal of $M$. The corresponding anisotropic normal of $M$ is defined by
\begin{equation*}
\nu_F=\n F(\nu).
\end{equation*}
The anisotropic principal curvatures $\k_F=(\kappa_{1}^{F},\ldots,\kappa_{n-1}^{F})\in\R^{n-1}$ are  defined as  the eigenvalues of the map
\[
d\nu_{F}\colon T_{x}M \to T_{\nu_{F}(x)}\p \mathcal W.
\]
The mean curvature (with respect to $\nu$) is defined to be $$H_F=\s_1(\kappa_F)=\sum_i\kappa_{i}^{F}.$$ The second mean curvature (with respect to $\nu$) is defined to be
$$\s_2(\kappa_F)=\sum_{i<j} \kappa_{i}^{F}\kappa_j^F.$$

A well-known variational characterization for $H_F$ is that it arises from the first variation
of the anisotropic area functional $|\p \O|_F$, (see e.g. Proposition 2.1 in \cite{Xia}), namely,
\begin{equation*}
  \frac{d}{dt}\int_{\p \O_t}F(\nu)d\sigma=\int_{\p \O_t}H_F\langle Y,\nu\rangle d\sigma,
\end{equation*}
where $Y$ is  the variation vector filed for $\p\O_t$.

In this paper we are interested in the case when $M$ is given by a regular level set of a smooth
function $u$, that is $M=\{u=t\}$ for some regular value $t$. For our purpose, we choose the unit normal $\nu=-\frac{\nabla u}{|\nabla u|}$ and
\begin{equation*}
\nu_F=-F_{\xi}(\nabla u), \ \ H_F=-{\rm div}(F_{\xi}(\nabla u)).
\end{equation*}
In this case, we have that \begin{equation}\label{h-f}
H_F=-{\rm div}(F_{\xi}(\nabla u))=-F_{ij}u_{ij},
\end{equation}
Here $\rm div$ is the Euclidean divergence. See e.g. \cite{DGX1}.

From Theorem 2.5 in \cite{DGX1}, we also have that
$$\s_2(\kappa_F)=S_2(F_{ik}u_{kj}).$$
It follows that
\begin{equation}\label{umbilic}
\begin{aligned}
\frac{n-2}{n-1}S_1^2(F_{ik}u_{kj})-2S_2(F_{ik}u_{kj})=&\frac{n-2}{n-1}\s_1^2(\kappa_F)-2\s_2(\kappa_F).
\end{aligned}
\end{equation}
We remark that the right hand side of \eqref{umbilic} is exactly the square of traceless part of anisotropic second fundamental form, which is nonnegative.

Next we give two propositions to the connection between an anisotropic $p$-harmonic function $u$ and the anisotropic curvatures of regular level sets of $u$. The first one is a formula for the anisotropic mean curvature of regular level set of $u$.
\begin{proposition}\label{prop:2.2'}
Let $u$ satisfy $\Delta_{F, p}u=0$.  Then the anisotropic mean curvature of regular level set of $u$ is given by
%$$F{\rm div}(F_{\xi}(\nabla u))+(p-1)\langle\nabla F(\nabla u),F_{\xi}(\nabla u)\rangle=0$$
 %we have
%
\begin{equation}\label{h-f-2}
H_F=(p-1)F^{-1}F_iF_j u_{ij}.
\end{equation}
\end{proposition}
\begin{proof}Note that $\Delta_{F, p}u=0$ implies $FF_{ij}u_{ij}+(p-1)F_iF_j u_{ij}=0$. The assertion follows from \eqref{h-f}.
\end{proof}
The second is a Kato-type identity for anisotropic $p$-harmonic functions, which generalizes Proposition 4.4 in \cite{FMP}.
\begin{proposition}\label{prop:2.2}
Let $u$ satisfy $\Delta_{F, p}u=0$. Let $\kappa_F$ denote the anisotropic principal curvature of regular level set of $u$. Then, at any point where $F(\nabla u)\not=0$, the following
identity holds true:
\begin{equation}\label{eqn:2.02}
\begin{aligned}
 a_{ij}a_{kl}u_{ik}u_{jl}=&F^2(\n u)\left[\frac{n-2}{n-1}\s_1^2(\kappa_F)-2\s_2(\kappa_F)\right]\\
 &+(1+\frac{(p-1)^2}{n-1})|\nabla (F(\nabla u))|^2_{a_F}\\
 &+(1-\frac{(p-1)^2}{n-1})|\nabla^{T} (F(\nabla u))|^2_{a_F},
 \end{aligned}
\end{equation}
where \begin{equation}\label{eqn:2.03}
|\nabla (F(\nabla u))|^2_{a_F}:=a_{ij}(F_ku_{ki})(F_lu_{lj})\ge 0,
\end{equation}
and
\begin{equation}\label{eqn:2.04}
|\nabla^T (F(\nabla u))|^2_{a_F}:=FF_{ij}(F_ku_{ki})(F_lu_{lj})\ge 0.
\end{equation}

%where $h_F$ and $H_F$ are respective the anisotropic second fundamental form the mean curvature of the level
%sets of $u$, and $g^T_F$ is the metric induced by on the level sets of $u$. Finally, for a given differential
%function f. we agree that $\nabla^Tf$ indicates the tangential part of the gradient.
\end{proposition}
\begin{remark}
$|\cdot|_{a_F}$ can be viewed as the norm with respect to the positive definite matrix $(a_{ij})$. Also by the definition, we have
\begin{eqnarray}\label{orth-decomp}
|\nabla (F(\nabla u))|^2_{a_F}=|\nabla^T (F(\nabla u))|^2_{a_F}+(F_iF_ju_{ij})^2.
\end{eqnarray}

 \end{remark}
\begin{proof} The proof is inspired by \cite{WX4}. Note that $a_{ij}=FF_{ij}+F_iF_j$, we can separate $\Delta_F u=a_{ij}u_{ij}=A+B$ where

\begin{equation*}
A=F_iF_ju_{ij}\ \ {\rm and }\ \ B=FF_{ij}u_{ij}.
\end{equation*}
Since\begin{equation*}
 a_{ij}a_{kl}u_{ik}u_{jl}=A^2+2FF_{kl}F_iF_ju_{ik}u_{jl}+F^2F_{ij}F_{kl}u_{ik}u_{jl},
\end{equation*}
and
\begin{equation*}
\frac{(A+B)^2}{n}+\frac n{n-1}\Big(\frac B{n}-\frac{n-1}n A\Big)^2=A^2+\frac1{n-1}B^2,
\end{equation*}%
we have
\begin{equation}\label{eqn:2.17}
\begin{aligned}
a_{ij}a_{kl}u_{ik}u_{jl}=&\frac{(A+B)^2}{n}+\frac n{n-1}\Big(\frac B{n}-\frac{n-1}n A\Big)^2-\frac1{n-1}B^2\\
&+2FF_{kl}F_iF_ju_{ik}u_{jl}+F^2F_{ij}F_{kl}u_{ik}u_{jl}
\\=&\frac{(a_{ij}u_{ij})^2}{n}+\frac n{n-1}\Big(\frac {FF_{ij}u_{ij}}{n}-\frac{n-1}n F_iF_ju_{ij}\Big)^2
\\&-\frac1{n-1}(FF_{ij}u_{ij})^2
+2FF_{kl}F_iF_ju_{ik}u_{jl}+F^2F_{ij}F_{kl}u_{ik}u_{jl}.
\end{aligned}
\end{equation}
On the on hand, it follows $\Delta_{F, p}u=0$ that
\begin{equation}\label{eqn:2.18}
a_{ij}u_{ij}=-(p-2)F_iF_ju_{ij}.
\end{equation}
On the other hand, we have the following computations
\begin{equation}\label{eqn:2.05}
 2FF_{kl}F_iF_ju_{ik}u_{jl}=2a_{kl}F_iF_ju_{ik}u_{jl}-2F_iF_jF_kF_lu_{ik}u_{jl}.
\end{equation}
Substituting \eqref{eqn:2.18} and \eqref{eqn:2.05} into \eqref{eqn:2.17} yields to
\begin{equation}\label{eqn:2.12}
\begin{aligned}
a_{ij}a_{kl}u_{ik}u_{jl}=&\Big(\frac{(p-1)^2}{n-1}-1\Big)(F_iF_ju_{ij})^2+2a_{kl}F_iF_ju_{ik}u_{jl}\\
&+F^2F_{ij}F_{kl}u_{ik}u_{jl}-\frac1{n-1}(FF_{ij}u_{ij})^2\\
=&\Big(\frac{(p-1)^2}{n-1}-1\Big)(F_iF_ju_{ij})^2+(1-\frac{(p-1)^2}{n-1})|\nabla (F(\nabla u))|^2_{a_F}\\
&+F^2F_{ij}F_{kl}u_{ik}u_{jl}-\frac1{n-1}(FF_{ij}u_{ij})^2\\
&+\left(1+\frac{(p-1)^2}{n-1}\right)|\nabla (F(\nabla u))|^2_{a_F}\\
=&\left(1-\frac{(p-1)^2}{n-1}\right)|\nabla^T (F(\nabla u))|^2_{a_F}+\left(1+\frac{(p-1)^2}{n-1}\right)|\nabla (F(\nabla u))|^2_{a_F}\\
&+F^2F_{ij}F_{kl}u_{ik}u_{jl}-\frac1{n-1}(FF_{ij}u_{ij})^2,
\end{aligned}
\end{equation}
where the last equality we have used \eqref{orth-decomp}.
%
%\begin{equation*}
%|\nabla F(\nabla u)|^2_{a_F}=|\nabla^T F(\nabla u)|^2_{a_F}+|\nabla_{\nu_F} F(\nabla u)|^2_{a_F}=
%|\nabla^T F(\nabla u)|^2_{a_F}+(F_iF_ju_{ij})^2.
%\end{equation*}
%
 Finally using \eqref{umbilic}, we obtain\begin{eqnarray*}
F_{ij}F_{kl}u_{ik}u_{jl}-\frac1{n-1}(F_{ij}u_{ij})^2&=&\frac{n-2}{n-1}S_1^2(F_{ik}u_{kj})-2S_2(F_{ik}u_{kj})\\&=&\frac{n-2}{n-1}\s_1^2(\kappa_F)-2\s_2(\kappa_F).
\end{eqnarray*}
This completed the proof of Proposition \ref{prop:2.2}.
\end{proof}

%
%Next we give the definition of the anisotropic perimeter of a set with respect to $F$ , a co-area
%formula and an isoperimetric inequality.
%

%For a domain $\Omega\subset\mathbb R^n$, a subset $E\subset\Omega$ and a function of bounded variation $u\in{\rm BV}(\Omega)$, let
%
%\begin{equation*}
%\int_{\Omega}|\nabla u|_F={\rm sup}\Big\{\int_{\Omega}u{\rm div}\sigma dx:\,\sigma\in C^1_0(\Omega;\mathbb R^n),\, F^0(\sigma)\leq 1\Big\}
%\end{equation*}
%
%be the anisotropic bounded variation of u with respect to $F$ and
%
%\begin{equation*}
%P_F(E):
%\end{equation*}

We also have some facts as following
\begin{proposition}
Let $u$ satisfy $\Delta_{F, p}u=0$. Then, at any point where $F(\nabla u)\not=0$, the following
identity holds true:
\begin{equation}\label{eqn:2.21}
\begin{aligned}
F^{2-p}a_{jk,p}F_iu_{ik}F_lu_{lj}=&|\nabla^TF(\nabla u)|^2_{a_F}+\frac{1}{p-1}H_F^2F^2(\nabla u), \end{aligned}
\end{equation}
\begin{equation}\label{eqn:2.22}
\begin{aligned}
a_{jk,p}a_{il,p}u_{ik}u_{lj}F^{4-2p}=&\frac{n}{n-1}H_F^2F^2+2(p-1)|\n^T (F(\n u))|^2_{a_F}\\
&+F^2(\n u)\left[\frac{n-2}{n-1}\s_1^2(\kappa_F)-2\s_2(\kappa_F)\right].
\end{aligned}
\end{equation}
\end{proposition}
\begin{proof}
From Proposition \ref{prop:2.2}, we have that
\begin{equation}\label{eqn:2.19}
\begin{aligned}
 a_{ij}a_{kl}u_{ik}u_{jl}=&F^2(\n u)\left[\frac{n-2}{n-1}\s_1^2(\kappa_F)-2\s_2(\kappa_F)\right]\\
 &+\Big(1+\frac{(p-1)^2}{n-1}\Big)\frac1{(p-1)^2}F^2H_F^2+2|\nabla^{T} (F(\nabla u))|^2_{a_F}.
 \end{aligned}
\end{equation}
By using \eqref{eqn:2.1} and \eqref{eqn:2.19}, directing computation gives that
\begin{equation*}
\begin{aligned}
a_{jk,p}a_{il,p}u_{ik}u_{lj}F^{4-2p}=&(a_{jk}+(p-2)F_jF_k)(a_{il}+(p-2)F_iF_l)u_{ik}u_{lj}\\
=&a_{jk}a_{il}u_{ik}u_{lj}+2(p-2)|\n (F(\n u))|^2_{a_F}+\frac{(p-2)^2}{(p-1)^2}H_F^2F^2\\
=&\frac{n}{n-1}H_F^2F^2+2(p-1)|\n^T (F(\n u))|^2_{a_F}\\
&+F^2(\n u)\left[\frac{n-2}{n-1}\s_1^2(\kappa_F)-2\s_2(\kappa_F)\right].
\end{aligned}
\end{equation*}
This completes the assertion.
\end{proof}

\

\section{Anisotropic $p$-capacitary potential}\label{sect:3}
%
%We recall the notion of anisotropic $p$-capacity
%introducing at the same time a normalised version of it that is suitable for our applications.
%
In this section, we will introduce the anisotropic $p$-capacity and its related properties.
\begin{definition}[anisotropic $p$-capacity]\label{def:2.1}\

Let $\Omega\subset\mathbb R^n$ be a bounded open set with smooth boundary.
%
%\begin{enumerate}
%\item The anisotropic $p$-capacity of $\Omega$ is defined as
%
%\begin{equation*}
%{\rm Cap}_{F, p}(\Omega)={\rm inf}\Big\{\int_{\mathbb R^n}F^p(\nabla u)dv\mid\,
%u\in \mathscr{C}^{\infty}_{c}(\mathbb R^n),\, u\geq1\ \ {\rm on}\ \ \Omega\Big\}.
%\end{equation*}
%
%\item
For $p\in [1, n)$, the anisotropic $p$-capacity of $\Omega$ is defined as
\begin{equation}\label{eqn:2.3}
{\rm Cap}_{F,p}(\Omega)={\rm inf}\Big\{
\int_{\mathbb R^n}F^p(\nabla u)dx\,\Big|\, u\in C_{c}^{\infty}(\mathbb R^n),\, u\geq1\ \ {\rm on}\ \ \Omega\Big\},
\end{equation}
%\end{enumerate}
%
%where $\kappa_n$ is the anisotropic perimeter of unit Wulff ball.
\end{definition}
It is easy to check that (see e.g. \cite{Maz}) \begin{equation}\label{eqn:6.4}
{\rm Cap}_{F, p}(\O)={\rm inf}\Big\{\int_{\mathbb R^n}F^{p}(\nabla u)\, dx\, \Big|\, u\geq\chi_{\Omega}, \, u\in{C}^{\infty}_{c}(\mathbb R^n)\Big\}.
\end{equation}

Next we consider the case $p\in (1, n)$.
The variational structure of the above definition leads naturally to the formulation of the problem
\begin{equation}\label{eqn:1.1}
\left\{
\begin{aligned}
\Delta_{F,p}u=&0\ \ {\rm in}\ \ \mathbb R^n\setminus\bar{\Omega}\\
u=&1\ \ {\rm on} \ \ \partial\Omega\\
u(x)\rightarrow& 0\ \ {\rm as}\ \ |x|\rightarrow\infty,
\end{aligned}\right.
\end{equation}
we say a function $u\in W^{1,p}_{loc}(U)$ is a
weak solution of $\Delta_{F,p}u=0$ in an open set $U$ if
\begin{equation*}
\int_U\langle F^{p-1}(\nabla u)F_{\xi}(\nabla u),\nabla\psi\rangle dx=0.
\end{equation*}
for any $\psi\in C^{\infty}_{c}(U)$.
A weak solution satisfying \eqref{eqn:1.1} is called the {\it anisotropic $p$-capacitary potential} associated with $\Omega$.
%It is well known that, for every bounded open set $\Omega$ with smooth boundary and every $1<p<n$,
%problem \eqref{eqn:1.1} admits a unique weak solution. Such a solution is called the anisotropic $p$-capacitary
%potential associated with $\Omega$.  By the important contributions \cite{CNSXYZ,DGX}, we know
%that weakly $p$-harmonic functions are $\mathscr{C}_{loc}^{ 1,\alpha}$.
%

We collect the existence and regularity results for the anisotropic $p$-capacity potentials in the following theorem, see e.g. \cite{BC, BCS}.
\begin{theorem}[Existence and regularity of anisotropic $p$-capacity potentials]\

Let $\Omega$ be a bounded
open set with smooth boundary, and let $1<p<n$. Then
\begin{enumerate}
\item There exists a unique weak solution $u\in C^{1,\alpha}_{loc}(\mathbb R^n\setminus\bar{\Omega})
\cap C(\mathbb R^n\setminus\Omega)$ to \eqref{eqn:1.1};
\item $u$ is smooth away from ${\rm Crit}(u):=\{\n u=0\}$;
\item The solution u fullfills
\begin{equation}\label{eqn:2.4}
{\rm Cap}_{F, p}(\Omega)=
\int_{\mathbb R^n\setminus\bar{\Omega}}F^p(\nabla u)dx.
\end{equation}
%
%where ${\rm Cap}_{F, p}(\Omega)$ is the normalised $p$-capacity of $\Omega$ defined in \eqref{eqn:2.3}.
\end{enumerate}
\end{theorem}

\begin{remark}\label{rem:3.1}Note that since $\partial\Omega$ is assumed to be smooth, by the Hopf Lemma for
anisotropic $p$-harmonic functions, we have that $\nabla u\not=0$ in a neighborhood of  $\p\O$.
In particular, $u$ is smooth in $\p\O$. Coupled with this fact, the asymptotic
expansions below imply that ${\rm Crit}(u)=\big\{x\in\mathbb R^n\setminus\bar{\Omega}\,|\,Du(x)=0\big\}$
is a compact subset of $\mathbb R^n\setminus\bar{\Omega}$
(generically depending on $p$), and in turn that $u$ is analytic outside this set.
\end{remark} %Finally, the set ${\rm Crit}(u)$ is a priori allowed to have full measure.

For $1<p<n$, let $$\Gamma_{F, p}(x)=\frac{p-1}{n-p}\left(\frac1{\kappa_{n-1}}\right)^{\frac{1}{p-1}}F^o(x)^{\frac{p-n}{p-1}}.$$
One can check that $$\Delta_{F,p}\Gamma_{F, p}(x)=\delta_0\hbox{ in }\rr^n,$$ where $\delta_0$ is the Dirac Delta function about the origin. We call
$\Gamma_{F, p}$ the fundamental solution to $\Delta_{F,p}u=0$ in $\rr^n$. See \cite{WX3}.

\begin{proposition}\label{lem:2.1}\

Let $\Omega\subset\mathbb R^n$ be a bounded open set with
smooth boundary, and $1<p<n$. Let $u$ is a
weak solution of $\Delta_{F,p}u=0$ in $\rr^n\setminus \O$. Then $u$ satisfies
\begin{enumerate}
\item $$\lim_{|x|\rightarrow+\infty}\frac{u(x)}{\Gamma_{F,p}(x)}={{\rm Cap}_{F, p}}(\Omega)^{\frac1{p-1}},$$
\item $$\nabla u(x)=\,{{\rm Cap}_{F, p}}(\Omega)^{\frac1{p-1}}\nabla \Gamma_{F,p}(x)+o(|x|^{-\frac{n-1}{p-1}}),\; {\rm as} \; |x|\to +\infty .$$
\end{enumerate}
%where ${\rm C}_{F, p}(\Omega)$ is the normalised anisotropic $p$-capacity of $\Omega$ defined in \eqref{eqn:2.3}.
%In particular, ${\rm Crit}(u)$ is a  compact subset of $\mathbb R^n\setminus\bar{\Omega}$, possibly with full measure.
\end{proposition}
\begin{proof}
If  $u$ is solution of \eqref{eqn:1.1}, it is a standard argument by using comparison theorem to show that there exists two positive constants $C_1, C_2$ such that
$$C_1 \Gamma_{F, p}\le u\le C_2 \Gamma_{F, p}.$$

Following the argument of \cite{KV}, Theorem 1.1 and Remark 1.5, (see \cite{WX3}, Theorem 4.1 and Remark 4.1 for anisotropic case), we conclude that there exists  $\gamma\in \rr$ such that
\begin{eqnarray}\label{eqn:2.5}
&&\lim_{|x|\rightarrow+\infty}\frac{u(x)}{\Gamma_{F,p}(x)}=\gamma, \\
&&\lim_{|x|\rightarrow+\infty}(F^o(x))^{\frac{n-1}{p-1}}\left(\nabla u-\gamma\nabla \Gamma_{F,p}\right)=0.
\end{eqnarray}
%
%where $\gamma$ is constant and $\Gamma(x)=\frac{p-1}{n-p}\kappa_{n-1}^{-\frac1{p-1}}(F^o(x))^{\frac{p-n}{p-1}}$ is fundamental
%solution of \eqref{eqn:1.1}.
%
It remains to show
 $$\gamma={\rm Cap}_{F, p}(\O)^{\frac1{p-1}}.$$
Through the following integration by parts that
holds true by the anisotropic $p$-harmonicity of $u$:
\begin{equation}\label{eqn:2.6}
\int\limits_{\partial\Omega}F^{p-1}(\nabla u)F(\nu)d\sigma=-
\lim_{R\rightarrow\infty}\int\limits_{\partial \mathcal{W}_R} F^{p-1}(\nabla u)\langle F_{\xi}(\nabla u),\nu_{\p \mathcal{W}_R}\rangle d\sigma,
\end{equation}
where $\nu=-\frac{\n u}{|\n u|}$ on $\p \O$ and $\nu_{\p \mathcal{W}_R}$ is the outward unit normal to $\mathcal{W}_R$,
we shall see in Proposition \ref{lem:2.2} below that the left-hand side of \eqref{eqn:2.6} is
exactly ${\rm Cap}_{F, p}(\O)$.
We next compute the limit in the right-hand side of  \eqref{eqn:2.6}.
Since $$\nabla u= \gamma \n \Gamma_{F, p}+o(|x|^{-\frac{n-1}{p-1}}),$$
we have
$$F(\nabla u)= \gamma F(\n \Gamma_{F,p})+o(|x|^{-\frac{n-1}{p-1}})= \gamma \left(\frac1{\kappa_{n-1}}\right)^{\frac{1}{p-1}}F^o(x)^{-\frac{n-1}{p-1}}+o(|x|^{-\frac{n-1}{p-1}})$$
On $\partial \mathcal{W}_R$,
 $$\nu_{\p \mathcal{W}_R}=\frac{\n F^o}{|\n F^o|}=-\kappa_{n-1}^{\frac{1}{p-1}}F^o(x)^{\frac{n-1}{p-1}}\frac{\n \Gamma_{F, p}}{|\n F^o|}=-\kappa_{n-1}^{\frac{1}{p-1}}F^o(x)^{\frac{n-1}{p-1}}\frac{\g^{-1}\n u}{|\n F^o|}+o(1).$$
It follows that on $\partial \mathcal{W}_R$
\begin{eqnarray*}
\langle F_{\xi}(\nabla u),\nu_{\p \mathcal{W}_R}\rangle&=&-\kappa_{n-1}^{\frac{1}{p-1}}F^o(x)^{\frac{n-1}{p-1}}\frac{\g^{-1}F(\n u)}{|\n F^o|}+o(1)
\\&=&-\frac{1}{|\n F^o|}+o(1).
\end{eqnarray*}
Hence
\begin{eqnarray*}
F^{p-1}(\nabla u)\langle F_{\xi}(\nabla u),\nu_{\p \mathcal{W}_R}\rangle=-\g^{p-1}\frac1{\kappa_{n-1}}\frac{(F^o(x))^{1-n}}{|\n F^o|}+o(|x|^{1-n}).
\end{eqnarray*}
Combining the fact that
$$\int\limits_{\partial \mathcal{W}_R} \frac{(F^o(x))^{1-n}}{|\n F^o|}d\sigma=\kappa_{n-1},$$ we deduce that
\begin{eqnarray*}
\lim_{R\rightarrow\infty}\int\limits_{\partial \mathcal{W}_R} F^{p-1}(\nabla u)\langle F_{\xi}(\nabla u),\nu_{\p \mathcal{W}_R}\rangle d\sigma=-\g^{p-1}.
\end{eqnarray*}
Thus $\gamma={\rm Cap}_{F, p}(\O)^{\frac1{p-1}}$, and
$(1), (2)$ of Proposition \ref{lem:2.1} follows.
\end{proof}
We need the following  expression for ${\rm Cap}_{F, p}(\Omega)$ in terms of an integral on $\partial\Omega$.
\begin{proposition}\label{lem:2.2}
Let $\Omega\subset\mathbb R^n$ be a bounded open set with smooth boundary, and let $1<p<n$.
Then, the solution $u$ to \eqref{eqn:1.1} satisfies
\begin{equation}\label{eqn:2.7}
{\rm Cap}_{F, p}(\Omega)=
\int_{\partial{\Omega}}F^{p-1}(\n u)F(\nu)d\sigma.
\end{equation}
where $\nu=-\frac{\n u}{|\n u|}$ on $\p \O$.
 %where ${\rm Cap}_{F, p}(\Omega)$ is the normalised $p$-capacity of $\Omega$ defined in \eqref{eqn:2.3}.
\end{proposition}
\begin{proof}
For $\varepsilon>0$, let $U_{\varepsilon}$ be the $\varepsilon$-tubular neighborhood of Crit$(u)$, namely
\begin{equation*}
U_{\varepsilon}=\Big\{x\in\mathbb R^n\setminus\bar{\Omega}\,\big|\, {\rm dist}(x,{\rm Crit(u)})<\varepsilon\Big\},
\end{equation*}
where dist$(x,{\rm Crit}(u))$ is the  distance of $x$ from ${\rm Crit}(u)$. By the compactness of Crit$(u)$
in $\mathbb R^n\setminus\bar{\Omega}$, we have that $U_{\varepsilon}\subset\{u\geq t\}$, for $\varepsilon > 0$ and $t > 0$ small enough. Since  $F(\nabla u)=0$ on
Crit$(u)$, we have from identity \eqref{eqn:2.4} and by Monotone Convergence Theorem that
\begin{equation*}
{\rm Cap}_{F, p}(\Omega)=\lim_{\varepsilon\rightarrow0^{+}}\lim_{t\rightarrow0^{+}}
\int\limits_{\{u\geq t\}\setminus U_{\varepsilon}}F^p(\nabla u)dx.
\end{equation*}
Then, for $\varepsilon$ and $t$ small enough, by intergration by parts,
\begin{equation*}
\begin{aligned}
\int\limits_{\{u\geq t\}\setminus U_{\varepsilon}}F^p(\nabla u)dx=&\int\limits_{\{u\geq t\}\setminus U_{\varepsilon}}
{\rm div}(uF^{p-1}(\nabla u)F_{\xi}(\nabla u))dx\\
=& -t\int\limits_{\{u= t\}}\frac{F^{p}(\nabla u)}{|\nabla u|}d\sigma+\int\limits_{\partial\Omega}F^{p-1}(\nabla u)F(\nu)d\sigma\\
&+\int\limits_{\partial U_{\varepsilon}}uF^{p-1}(\nabla u)\langle F_{\xi}(\nabla u),\nu_{\p U_{\varepsilon}}\rangle d\sigma.
\end{aligned}
\end{equation*}
%Observe that $\nu_{V_{\varepsilon}}$ is well defined almost everywhere on $\partial {V_{\varepsilon}}$ and for almost every $\varepsilon$, in view of the Sard-type property for Lipschitz functions proved in \cite{ABC}.
 By the computation in the proof of Proposition \ref{lem:2.1},  we see $\int\limits_{\{u= t\}}\frac{F^{p}(\nabla u)}{|\nabla u|}d\sigma$ is bounded when $t\to 0+$, which implies
 $$-t\int\limits_{\{u= t\}}\frac{F^{p}(\nabla u)}{|\nabla u|}d\sigma\to 0\hbox{ as } t\to 0.$$ On the other hand, we have $$ \int\limits_{\partial U_{\varepsilon}}uF^{p-1}(\nabla u)\langle F_{\xi}(\nabla u),\nu_{\p U_{\varepsilon}}\rangle d\sigma\to 0\hbox{ as } \varepsilon\to 0.$$since $F(\nabla u)$ tends to zero as $x$ approaches Crit$(u)$. The assertion follows.
\end{proof}

\

\section{Essential monotonicity}\label{sect:4}
In this section, we assume $u$ is anisotropic $p$-harmonic, that is, $\Delta_{F, p}u=0$. We will establish some Lemmas for later using.

Let $$
p^*:=\frac{(n-1)(p-1)}{n-p}.
$$ and \begin{equation*}
  \Lambda=\Big\{(p,q)\in\mathbb R^2\Big|\,1<p<n,\ {\rm and} \ q\geq1+\frac{1}{p^*}\Big\}.
\end{equation*}

For $(p, q)\in \Lambda$, we consider the function $\Phi_{p,q}:[1,+\infty)\rightarrow\mathbb R^n$ given by
\begin{equation}\label{eqn:4.1}
\Phi_{p,q}(\tau):=\tau^{(q-1)p^*}\int\limits_{\{u=1/\tau\}}F^{q(p-1)}(\nabla u)F(\nu)d\sigma.
\end{equation}

By the asymptotic behavior given in Lemma \ref{lem:2.1}, it is direct to check the limits of $\Phi_{p,q}(\tau)$.
\begin{lemma}\label{lem:4.0}
%Let $\Phi_{p,q}:[1,+\infty)\rightarrow\mathbb R$ be e defined as in \eqref{eqn:4.1}. Then
%
\begin{equation}\label{eqn:4.2}
\lim_{\tau\rightarrow+\infty}\Phi_{p,q}(\tau)=\big(\frac{n-p}{p-1}\big)^{(q-1)p^*}\big(\kappa_{n-1}\big)^{\frac{(q-1)(p-1)}{n-p}}({\rm Cap}_{F, p}(\Omega))^{1-\frac{(q-1)(p-1)}{n-p}}.
\end{equation}
\end{lemma}
\begin{proof}
From the asymptotic behavior of $u$ in Proposition \ref{lem:2.1}, we know that
For $\varepsilon>0$, there exists $\tau_0$ such that for $\tau>\tau_0$  and $u(x)=\frac{1}{\tau}$,
\begin{eqnarray*}
&&(1-\varepsilon)\gamma \Gamma_{F,p}\le u(x)\le (1+\varepsilon)\gamma \Gamma_{F,p},
\\&&(1-\varepsilon)\gamma \left(\frac1{\kappa_{n-1}}\right)^{\frac{1}{p-1}}F^o(x)^{-\frac{n-1}{p-1}} \le F(\nabla u(x))\le (1+\varepsilon)\gamma \left(\frac1{\kappa_{n-1}}\right)^{\frac{1}{p-1}}F^o(x)^{-\frac{n-1}{p-1}}.
\end{eqnarray*}
From this one deduces that on such $\{u(x)=\frac{1}{\tau}\}$
\begin{eqnarray*}
&&(1-\tilde{\varepsilon})\gamma^{-\frac{(q-1)(p-1)^2}{n-p}} \big(\frac{n-p}{p-1}\big)^{(q-1)p^*}\big(\kappa_{n-1}\big)^{\frac{(q-1)(p-1)}{n-p}}\tau^{-(q-1)p^*}
\\&& \le F^{(q-1)(p-1)}(\nabla u(x))
\\&&\le (1+\tilde{\varepsilon})\gamma^{-\frac{(q-1)(p-1)^2}{n-p}} \big(\frac{n-p}{p-1}\big)^{(q-1)p^*}\big(\kappa_{n-1}\big)^{\frac{(q-1)(p-1)}{n-p}}\tau^{-(q-1)p^*}.
\end{eqnarray*}
where $\tilde{\varepsilon}$ is some function on $\varepsilon$ with the property $\lim_{\varepsilon\to 0}\tilde{\varepsilon}=0.$
Arguing as in Proposition \ref{lem:2.2}, we see that \begin{eqnarray*}
\int\limits_{\{u=1/\tau\}}F^{p-1}F(\nu)d\sigma={\rm Cap}_{F, p}(\Omega).
\end{eqnarray*}
Since $$\tau^{(q-1)p^*}\int\limits_{\{u=1/\tau\}}F^{q(p-1)}F(\nu)d\sigma=\tau^{(q-1)p^*}
\int\limits_{\{u=1/\tau\}}F^{(q-1)(p-1)}(F^{p-1}F(\nu)d\sigma$$
It follows that
\begin{eqnarray*}
&&(1-\tilde{\varepsilon})\big(\frac{n-p}{p-1}\big)^{(q-1)p^*}\big(\kappa_{n-1}\big)^{\frac{(q-1)(p-1)}{n-p}}({\rm Cap}_{F, p}(\Omega))^{1-\frac{(q-1)(p-1)}{n-p}}
\\&\le&\tau^{(q-1)p^*}\int\limits_{\{u=1/\tau\}}F^{q(p-1)}F(\nu)d\sigma
\\&\le& (1+\tilde{\varepsilon})\big(\frac{n-p}{p-1}\big)^{(q-1)p^*}\big(\kappa_{n-1}\big)^{\frac{(q-1)(p-1)}{n-p}}({\rm Cap}_{F, p}(\Omega))^{1-\frac{(q-1)(p-1)}{n-p}}.
\end{eqnarray*}
The assertion follows by letting $\tau\to\infty$ and then $\varepsilon\to 0$.
\end{proof}

For  $(p,q)\in\Lambda$. Let us consider the vector field $X$ with component
\begin{equation}\label{eqn:4.3}
\begin{aligned}
X^j%=&(q-1)\Big\{u^{-{(q-1)p^*}+2}F^{(q-2)(p-1)-1}S^{ij}_2V_{\xi_i}(\nabla u)\\
%&+p^*u^{-{(q-1)p^*}+1}F^{q(p-1)}F_j\Big\}\\
=&-(q-1)u^{-{(q-1)p^*}+2}F^{q(p-1)-1}\left(a_{jk,p}F^{1-p}F_iu_{ik}-{p^*}\frac{FF_j}{u}\right)
\end{aligned}
\end{equation}
where
\begin{equation*}
\begin{aligned}
 a_{jk,p}=F^{p-2}(a_{jk}+(p-2)F_jF_k).
 \end{aligned}
\end{equation*}

Next, we have the following Lemma.
\begin{lemma}\label{lem:4.1}
$\Phi_{p,q}$ is differentiable at the regular values of $u$ and
\begin{equation}\label{eqn:4.4}
\begin{aligned}
\frac{d}{d\tau} \Phi_{p,q}(\tau)=\int\limits_{\{u=1/\tau\}}\langle X,\frac{\nabla u}{|\nabla u|}\rangle d\sigma=
\int\limits_{\{u<1/\tau\}}{\rm div}X dx.
\end{aligned}
\end{equation}
\end{lemma}
\begin{proof}
 By co-area formula and divergence theorem we have
\begin{equation}\label{eqn:3.3}
\begin{aligned}
\frac{d}{d\tau}\Phi_{p,q}(\tau)=&\frac{d}{d\tau}\Big\{\int\limits_{\{u=1/\tau\}}
u^{-{(q-1)p^*}}F^{q(p-1)}F(\nu)d\sigma\Big\}\\
=&\frac{d}{d\tau}\Big\{\int\limits_{\{u<1/\tau\}}\partial_{x_i}\left(u^{-{(q-1)p^*}}F^{(q-1)(p-1)}
F^{p-1}F_i\right)dx\Big\}\\
=&\frac{d}{d\tau}\Big\{\int\limits_{\{u<1/\tau\}}(q-1)u^{-{(q-1)p^*}}F^{q(p-1)}\left(-{p^*}\frac{F}{u}+(p-1)\frac{F_iF_ku_{ik}}{F}\right)dx\Big\}\\
=&-(q-1)\tau^{{(q-1)p^*}-2}\int\limits_{\{u=1/\tau\}}F^{q(p-1)-1}\left((p-1)\frac{F_iF_ku_{ik}}{F}-{p^*}\frac{F}{u}\right)F(\nu)d\sigma.
\end{aligned}
\end{equation}
By using the definition of $X$ and the fact
$$a_{jk,p}u_j=(p-1)F^{p-2}FF_k,$$ we deduce
\begin{equation*}
\begin{aligned}
\langle X,\,\nabla u\rangle=&-(q-1)u^{-{(q-1)p^*}+2}F^{q(p-1)-1}\left((p-1)F_kF_iu_{ik}-{p^*}\frac{F^2}{u}\right).
\end{aligned}
\end{equation*}
Then, we obtain the first equality.  The second equality follows from the divergence theorem,.
\end{proof}
\vskip 2mm

Finally, we compute the divergence of $X$.
\begin{lemma}[Divergence of $X$]\label{lem:4.2}
For any $(p,q)\in\Lambda$, let $u$ be a solution of \eqref{eqn:1.1} and $X$ be the vector field defined
in \eqref{eqn:4.3}. Then, the following identity holds at any point $x\in\mathbb R^n\setminus\Omega$ such that $F(\nabla u)\not=0$.
\begin{equation}\label{eqn:4.5}
\begin{aligned}
{\rm div}X=&-(q-1)u^{-{(q-1)p^*}+2}F^{q(p-1)-1}(\nabla u)\Big\{
\left[\frac{n-2}{n-1}\s_1^2(\kappa_F)-2\s_2(\kappa_F)\right]\\
&+\Big(q(p-1)-1\Big)\frac{|\nabla^T(F(\nabla u))|^2_{a_F}}{F^2(\n u)}\\
&+\Big(q-1-\frac{1}{p^*}\Big)\left|H_F-{p^*}\frac{F(\n u)}u\right|^2\Big\}\leq0.
\end{aligned}
\end{equation}
\end{lemma}
\begin{proof}Using \eqref{eqn:2.9} and $\Delta_{F, p}u=0$,
we can write $X$ in the following way
\begin{equation}\label{eqn:4.3'}
\begin{aligned}
X^j=(q-1)\Big\{&u^{-{(q-1)p^*}+2}F^{(q-2)(p-1)-1}S^{ij}_2V_{\xi_i}(\nabla u)\\
&+p^*u^{-{(q-1)p^*}+1}F^{(q-1)(p-1)}F^{p-1}F_j\Big\}.
\end{aligned}
\end{equation}
Separating
\begin{equation}\label{eqn:3.4}
{\rm div} X=(q-1)(I+II),
\end{equation}
where
\begin{equation*}
\begin{aligned}
I=&\partial_{x_j}\Big(u^{-{(q-1)p^*}+2}F^{(q-2)(p-1)-1}
S^{ij}_2V_{\xi_i}(\nabla u)\Big),\\
II=&{p^*}\partial_{x_j}\Big(u^{{-(q-1)p^*}+1}F^{(q-1)(p-1)}F^{p-1}F_j\Big),
\end{aligned}
\end{equation*}
%
%where
%\begin{equation*}
%  \ S^{ij}_2(W)V_{\xi_i}(\nabla u)=-F^{p-1}(\nabla u)a_{jk,p}(\nabla u)u_{ki}F_i.
%\end{equation*}
%
By using Proposition \ref{prop:2.2'}, \eqref{eqn:2.11} and \eqref{h-f-2}, we have the following
computations
\begin{equation}\label{eqn:3.5}
\begin{aligned}
I=&-(2-{(q-1)p^*})u^{-{(q-1)p^*}+1}F^{q(p-1)}H_F-u^{-{(q-1)p^*}+2}F^{q(p-1)-3}\times\\ &\times\Big[\big((q-2)(p-1)-1\big)F^{2-p}a_{jk,p}F_iF_lu_{ki}u_{lj}+F^{4-2p}a_{jk,p}a_{il,p}u_{ik}u_{lj}\Big],
\end{aligned}
\end{equation}
and
\begin{equation}\label{eqn:3.6}
\begin{aligned}
II=\Big(1-{(q-1)p^*}\Big){p^*}u^{-(q-1){p^*}}F^{q(p-1)+1}+{{(q-1)p^*}}u^{1-{(q-1)p^*}}F^{q(p-1)}{H_F}.
\end{aligned}
\end{equation}
%
%where
%\begin{equation*}
%H_F=(p-1)\frac{F_iF_ju_{ij}}{F(\nabla u)},\ \ a_{ij,p}=F^{p-2}(a_{ij}+(p-2)F_iF_j).
%\end{equation*}

Next, by using \eqref{eqn:2.21} and \eqref{eqn:2.22}, we have
\begin{equation}\label{eqn:3.7}
\begin{aligned}
 &\Big[(q-2)(p-1)-1\Big]F^{2-p}a_{jk,p}F_iF_lu_{ki}u_{lj}+F^{4-2p}a_{jk,p}a_{il,p}u_{ik}u_{lj}\\
 =&\Big[(q-2)(p-1)-1\Big]\Big(|\n^T (F(\n u))|^2_{a_F}+\frac{1}{p-1}H_F^2F^2\Big)+\frac{n}{n-1}H_F^2F^2\\
 &+2(p-1)|\n^T (F(\n u))|^2_{a_F}+F^2\left[\frac{n-2}{n-1}\s_1^2(\kappa_F)-2\s_2(\kappa_F)\right]\\
 =&\Big[(q-1)-\frac{1}{p^*}\Big]F^2H^2_F+\Big(q(p-1)-1\Big)|\nabla^T(F(\nabla u))|^2_{a_F}\\
 &+F^2\left[\frac{n-2}{n-1}\s_1^2(\kappa_F)-2\s_2(\kappa_F)\right].
\end{aligned}
\end{equation}
%
%where we have used
%
%\begin{equation*}
%\begin{aligned}
%a_{il}F_jF_ku_{ik}u_{lj}=&|\nabla^T F(\nabla u)|^2_{a_F}+|\nabla_{\nu_F} F(\nabla u)|^2_{a_F}\\
%=&
%|\nabla^T F(\nabla u)|^2_{a_F}+\frac{1}{(p-1)^2}F^2(\nabla u)H^2_F.
%\end{aligned}
%\end{equation*}
%

Substituting \eqref{eqn:3.5}-\eqref{eqn:3.7} into \eqref{eqn:3.4} yields \eqref{eqn:4.5}.
\end{proof}

\

%%%%%%%%%%%%%%%%%%%%%%%%%%%%%%%%%%%%%%%%%%%%%%%%%%%%%%%%%%%%%%%%%%%%%%%%%%%%%%%%%%%
\section{Effective monotonicity and $L^p$ anisotropic Minkowski inequality}\label{sect:5}

The aim of this section is to given a complete proof of Theorem \ref{thm:1.1}.  For this purpose,
we will establish the following two effective monotonicity inequalities:
\begin{equation*}
\Phi'_{p,q}(1)\leq 0\ \ {\rm  and} \ \ \Phi_{p,q}(+\infty)=\lim\limits_{\tau\to\infty} \Phi_{p,q}(\tau)\leq  \Phi_{p,q}(1).
\end{equation*}

\subsection{First effective inequality}\

Combining Lemma \ref{lem:4.1} and \ref{lem:4.2}, we see that $\Phi'_{p,q}(s)\leq \Phi'_{p,q}(S)$ for any $s<S$ if ${\rm Crit}(u)=\emptyset$.
%\begin{equation*}
%  \int\limits_{\{u=1/s\}}\langle X,\frac{\nabla u}{|\nabla u|}\rangle\leq
%  \int\limits_{\{u=1/S\}}\langle X,\frac{\nabla u}{|\nabla u|}\rangle.
%\end{equation*}
However, in general ${\rm Crit}(u)\not=\emptyset$.
Nevertheless, we are  able to provide an effective version of the considered monotonicity,
showing that  is actually in force, provided $S$ is large enough and $s$ is closed to $1$.
This is enough to get the desired effective inequality $\Phi'_{p,q}(1)\leq0$.

\begin{theorem}[Effective Monotonicity Formula--I]\label{thm:4.1}
For $(p,q)\in\Lambda$,  let $u$ be the
solution to \eqref{eqn:1.1}  and let $1< \bar s < \bar S <+\infty$
be such that ${\rm Crit}(u)\subset\{\bar S^{-1} < u < \bar s^{-1}\}$. Then, for every $1 \leq s \leq \bar s \leq \bar S \leq S$,
the inequality
\begin{equation}\label{eqn:4.6}
  \Phi'_{p,q}(s)\leq \Phi'_{p,q}(S),
\end{equation}
holds true, where $\Phi_{p,q}$ is defined in \eqref{eqn:4.1}. Moreover, one has that $\Phi'_{p,q}(1)\leq0$.
\end{theorem}
\begin{remark} The existence of $\bar s$  and $\bar S$ follows from Remark \ref{rem:3.1} and Proposition \ref{lem:2.1}.
\end{remark}
\begin{proof}
Fix $s$ and $S$ such that $1 \leq s \leq \bar s \leq \bar S \leq S$. For given $\varepsilon>0$, let  $\chi:[0,+\infty)\rightarrow\mathbb R$ be a smooth nonnegative cut-off-function,
such that
\begin{equation}\label{eqn:4.7}
\left\{
\begin{aligned}
\chi(t)=&0\ \ {\rm in} \ \ t<\frac12{\varepsilon},\\
{\chi}'(t)\geq&0\ \ {\rm in} \ \ \frac12{\varepsilon}\leq t\leq\frac32\varepsilon,\\
\chi(t)=&1\ \ {\rm in}\ \ t>\frac32\varepsilon.
\end{aligned}\right.
\end{equation}
Define a smooth vector field
\begin{equation*}
\tilde X=\chi\big(u^{-{(q-1)p^*}}F^{(q-1)(p-1)}(\nabla u)\big)X.
\end{equation*}
Since $F(\nabla u)=0$ on Crit$\,(u)$, we have  $X=0\hbox{ in }{\rm Crit}(u)$.
By choosing $\varepsilon$ small enough, we can make sure that
$$\tilde{X}=X\hbox{ on }\{u=s^{-1}\}\hbox{ and }
\{u=S^{-1}\},$$ since ${\rm Crit}(u)\subset\{\bar S^{-1} < u < \bar s^{-1}\}\subset \{S^{-1}<u<s^{-1}\}$.
Denote
\begin{equation*}
\begin{aligned}
 \Theta:=&\langle X,\nabla (u^{-{(q-1)p^*}}F^{(q-1)(p-1)}(\nabla u))\rangle\\
 =&(q-1)(p-1)u^{-{(q-1)p^*}}F^{(q-1)(p-1)-1}\Big\langle X,\nabla\big( F(\nabla u)\big)-\frac{n-1}{n-p}\frac{\nabla u}{u} F\Big\rangle.
 \end{aligned}
\end{equation*}
Then we have the following computations
\begin{equation}\label{eqn:4.8}
\begin{aligned}
&\int\limits_{\{u=1/s\}}\langle X,\frac{\nabla u}{|\nabla u|}\rangle d\sigma-\int\limits_{\{u=1/S\}}\langle X,\frac{\nabla u}{|\nabla u|}\rangle d\sigma=\int\limits_{\{1/S<u<1/s\}}{\rm div}{\tilde X}dx\\
=&\int\limits_{\{1/S<u<1/s\}\setminus U_{\varepsilon/2}}\chi{\rm div} Xdx+\int\limits_{U_{3\varepsilon/2}\setminus U_{\varepsilon/2}}{\chi}'\Theta dx.
\end{aligned}
\end{equation}
where in the last identity we have used the tubular neighborhood of Crit$\,(u)$ defined for every
$\delta> 0$ as $U_{\delta}=\{u^{-{(q-1)p^*}}F^{(q-1)(p-1)}(\nabla u) \leq\delta\}$.

From Lemma \ref{eqn:4.2}, we can see that the first term  on the right hand side of \eqref{eqn:4.8} is non-positive, we next prove that $\Theta\leq0$. By using \eqref{eqn:4.3}, \eqref{h-f-2} and \eqref{eqn:2.04}, we have that
\begin{equation}\label{eqn:4.9}
  \begin{aligned}
 \Theta
=&-(q-1)^2(p-1)u^{-2{(q-1)p^*}+2}F^{(2q-1)(p-1)-2}\times\\& \times\Big\{\frac{a_{jk,p}F_iu_{ik}}{F}F^{2-p}
-{p^*}\frac{FF_j}{u}\Big\}
\Big(F_lu_{lj}-\frac{(n-1)}{n-p}\frac{F}{u}u_j\Big)\\
=&-(q-1)^2(p-1)u^{-2{(q-1)p^*}+2}F^{(2q-1)(p-1)-1}\times\\& \times\Big\{\frac{a_{jk,p}F_iu_{ik}F_lu_{lj}}{F^p(\nabla u)}
+\frac{(n-1)^2(p-1)}{(n-p)^2}\frac{F^2}{u^2}
-2{p^*}\frac{F_kF_iu_{ik}}{u}\Big\}\\
=&-(q-1)^2u^{-2{(q-1)p^*}+2}F^{(2q-1)(p-1)-1}\times\\& \times\Big\{\Big|H_F-{p^*}\frac{F}{u}\Big|^2+(p-1)\frac{|\nabla^T(F(\nabla u))|^2_{a_F}}{F^2}\Big\}\leq0.
  \end{aligned}
\end{equation}
This completes the proof of \eqref{eqn:4.6}.

It follows  from \eqref{eqn:4.6} that, for every $S\geq \bar S$,
\begin{equation*}
\Phi'_{p,q}(1)\leq \Phi'_{p,q}(S).
\end{equation*}
Integrating both sides of the above inequality on an interval of the form $(\bar S,S)$, we obtain
\begin{equation*}
\Phi'_{p,q}(1)(S-\bar S)+\Phi_{p,q}(\bar S)\leq \Phi_{p,q}(S)
\end{equation*}
If by contradiction, $\Phi'_{p,q}(1)>0$, then, letting $S\rightarrow+\infty$ above inequality, we would deduce
$\lim_{S\rightarrow+\infty}\Phi_{p,q}(S)\rightarrow+\infty$  against the boundedness of $\Phi_{p,q}$ by Lemma \ref{lem:4.0}.
\end{proof}

\subsection{Second effective inequality}$\Phi_{p,q}(\infty)\leq \Phi_{p,q}(1)$.\

%
%There is no hope for deducing the global inequality $\Phi_{p,q}(+\infty)\leq \Phi_{p,q}(1)$
% from the pointwise inequality $\Phi_{p,q}(\tau)\leq 0$
%through integration, since the latter inequality may fail to be true-or even well defined-for
%too many values of $\tau\in[1,+\infty)$. To face the main difficulty of our program, we craft a new
%family of effective monotonicity formulas.
For a given $(p,q)\in\Lambda$ and a given $0 < \lambda < 1$, we
consider the vector field
\begin{equation}\label{eqn:4.10}
Y_{\lambda}=(u^{-1}-\lambda)X-F^{q(p-1)}u^{-(q-1)p^*}\n_{\xi}F(\n u),
\end{equation}
%
%where $X$ has been defined in \eqref{eqn:4.3}.
 It is convenient to observe that at a regular value of
$u$ it holds
\begin{equation}\label{eqn:4.11}
(\tau-\lambda)\Phi'_{p,q}(\tau)-\Phi_{p,q}(\tau)=\int\limits_{\{u=1/\tau\}}\Big\langle Y_{\lambda},\frac{\nabla u}{|\nabla u|}\Big\rangle
\end{equation}

Next we compute the divergence of $Y_{\lambda}$.
\begin{lemma}[Divergence of $Y_{\lambda}$]\label{lem:4.3}
For any $(p,q)\in\Lambda $ and any $0 < \lambda< 1$, let $u$ be the solution of \eqref{eqn:1.1} and $Y_{\lambda}$
be the vector field defined in \eqref{eqn:4.10}. Then, the following identity holds at any point $x\in\mathbb R^n\setminus\bar{\Omega}$ such that $F(\nabla u)\not=0$
\begin{equation*}
{\rm div}{Y_{\lambda}}=(u^{-1}-\lambda){\rm div}X\leq 0,
\end{equation*}
where ${\rm div}X$ is non-positive defined in \eqref{eqn:4.5}
\end{lemma}
\begin{proof}
By the very definition of $Y_{\lambda}$, we obtain that
\begin{equation*}
  \begin{aligned}
  {\rm div}Y_{\lambda}=(u^{-1}-\lambda){\rm div}X-u^{-2}\langle X,\,\nabla u\rangle-{\rm div}\big(F^{q(p-1)}u^{-(q-1)p^*} F_{\xi}(\nabla u)\big)
  \end{aligned}
\end{equation*}
By using the definition of $X$, we compute
\begin{equation*}
\begin{aligned}
u^{-2}\langle X,\,\nabla u\rangle=&-(q-1)u^{-{(q-1)p^*}}F^{q(p-1)-1}\left((p-1)F_kF_iu_{ik}-{p^*}\frac{F^2}{u}\right).
\end{aligned}
\end{equation*}
By using $\Delta_{F,p}u={\rm div}(F^{p-1}F_\xi(\nabla u))=0$, we get
\begin{equation*}
\begin{aligned}
&{\rm div}\big(F^{q(p-1)}u^{-(q-1)p^*} F_{\xi}(\nabla u)\big)=(q-1)u^{-{(q-1)p^*}}F^{q(p-1)-1}
\left((p-1)F_kF_iu_{ik}
-{p^*}\frac{F^2}{u}\right).
\end{aligned}
\end{equation*}
The assertion follows.
\end{proof}
%
%Again, in absence of critical points, the Divergence Theorem applied to the vector field $Y_{\lambda}$ on
%the open region $\{1/S < u < 1/s\}$ easily yields the inequality
%
%\begin{equation*}
%  \int\limits_{u=\frac1s}\langle Y_{\lambda},\, \frac{\nabla u}{|\nabla u|}\rangle d\sigma\leq \int\limits_{u=\frac1S}\langle Y_{\lambda},\, \frac{\nabla u}{|\nabla u|}\rangle d\sigma.
%\end{equation*}
%
%As usual, the difficult part is the treatment of the critical points.

\begin{theorem}[Effective Monotonicity Formula--II]\label{thm:4.2}
For any $1 < p < n$, let $u$ be the
solutions of \eqref{eqn:1.1}  and let $1< \bar s < \bar S <+\infty$
be such that ${\rm Crit}(u)\subset\{\bar S^{-1} < u < \bar s^{-1}\}$. Then, for every $1 \leq s \leq \bar s \leq \bar S \leq S$,
the inequality
\begin{equation}\label{eqn:4.12}
(s-\lambda)\Phi'_{p,q}(s)-\Phi_{p,q}(s)\leq (S-\lambda)\Phi'_{p,q}(S)-\Phi_{p,q}(S),
\end{equation}
holds true. Moreover, one has that $\Phi_{p,q}(+\infty)\leq \Phi_{p,q}(1)$.
\end{theorem}
\begin{proof}
Fix $s$ and $S$ such that $1 \leq s \leq \bar s \leq \bar S \leq S$. Let $\chi:[1,+\infty)\rightarrow\mathbb R$ be the same smooth nonnegative cut-off function as in the proof of
Theorem \ref{thm:4.1}, so that the properties  \eqref{eqn:4.7} holds. To simplify the notation, let us also set
\begin{equation*}
\eta_{\lambda}(u)=(u^{-1}-\lambda)^{-1}.
\end{equation*}
Then, let us consider the smooth vector field
\begin{equation*}
  \tilde Y_{\lambda}=\chi\Big(\eta_{\lambda}(u)u^{-{(q-1)p^*}}F^{(q-1)(p-1)}(\nabla u)\Big)Y_{\lambda},
\end{equation*}
where $Y_{\lambda}$ has been defined in \eqref{eqn:4.9}. Again, choosing $\varepsilon$ small enough, we can suppose $\tilde Y_{\lambda} = Y_{\lambda}$ on $\{u = 1/s\}$ and $\{u = 1/S\}$, with $s$ and $S$ as in the statement.

Denote
\begin{equation*}
\Psi:=\Big\langle \nabla\Big(\eta_{\lambda}(u)u^{-{(q-1)p^*}}F^{(q-1)(p-1)}(\nabla u)\Big), Y_{\lambda}\Big\rangle.
\end{equation*}

 By applying the divergence theorem to the smooth vector field $\tilde Y_{\lambda}$ on the region $\{1/S <u< 1/s\}$ gives
 \begin{equation*}
 \begin{aligned}
&\int\limits_{\{u=\frac1s\}}\langle Y_{\lambda},\, \frac{\nabla u}{|\nabla u|}\rangle d\sigma-\int\limits_{\{u=\frac1S\}}\langle Y_{\lambda},\, \frac{\nabla u}{|\nabla u|}\rangle d\sigma=\int\limits_{\{\frac1S<u<\frac1s\}}  {\rm div}\tilde Y_{\lambda}dx\\
&=\int\limits_{\{\frac1S<u<\frac1s\}\setminus U_{\varepsilon/2}}\chi\Big(\eta_{\lambda}(u)u^{-{(q-1)p^*}}F^{(q-1)(p-1)}(\nabla u)\Big){\rm div}Y_{\lambda}dx+\int\limits_{U_{3\varepsilon/2}\setminus U_{\varepsilon/2}}\chi' \Psi dx
\end{aligned}
 \end{equation*}
 where this time the tubular neighborhoods of ${\rm Crit}(u)$ are defined, for every $\delta > 0,$ as $U_{\delta} =
\{ \eta_{\lambda}(u)u^{-{(q-1)p^*}}F^{(q-1)(p-1)}(\nabla u)\leq\delta\}$. Since, as observed in Lemma \ref{lem:4.3}, the divergence of $Y_{\lambda}$ is non-positive on $\{1/S \leq u \leq 1/s\}\setminus U_{\varepsilon/2}$, where clearly $F(\nabla u)\not=0$. Next we prove
$\Psi$ is non-positive on $U_{3\varepsilon/2}\setminus U_{\varepsilon/2}$.

Noticing that
\begin{equation}\label{eqn:4.13}
\begin{aligned}
&\nabla \Big(\eta_{\lambda}(u)u^{-{(q-1)p^*}}F^{(q-1)(p-1)}(\nabla u)\Big)\\
=&u^{-2-{(q-1)p^*}}\eta_{\lambda}^2(u)F^{(q-1)(p-1)}\nabla u+(q-1)(p-1)\eta_{\lambda}\times\\
&\times u^{-{(q-1)p^*}}F^{(q-1)(p-1)-1}\Big(\nabla( F(\nabla u))-\frac{n-1}{n-p}\frac{F}{u}\nabla u\Big)
\end{aligned}
\end{equation}
By using Proposition \ref{prop:2.2'}, \eqref{eqn:4.9} and \eqref{eqn:4.13}, a direct computation gives
\begin{equation*}
\begin{aligned}
\Psi=&\Theta-2(q-1)\eta_{\lambda}(u)u^{-2{(q-1)p^*}}F^{(2q-1)(p-1)-1}\Big((p-1)F_iF_ku_{ik}-p^*\frac{F^2}{u}\Big)\\
&-\eta^2_{\lambda}(u)u^{-2{(q-1)p^*}}F^{(2q-1)(p-1)-1}\frac{F^2}{u^2}\\
=&-u^{-2{(q-1)p^*}}F^{(2q-1)(p-1)-1}\Big\{\Big[(q-1)u\Big(H_F-{p^*}\frac{F}{u}\Big)+\eta_{\lambda}(u)\frac{F}{u}\Big]^2+\\&+(q-1)^2(p-1)\frac{u^2}{F^2}|\nabla^T(F(\nabla u))|^2_{a_F}\Big\}
\leq 0,
\end{aligned}
\end{equation*}
%
%where we have used \eqref{eqn:4.9} and \eqref{eqn:4.13}.
%
This completes the proof of \eqref{eqn:4.12}.

It remains to show that $\lim_{\tau\rightarrow+\infty}\Phi_{p,q}(\tau)\leq \Phi_{p,q}(1)$.  Applying the inequality \eqref{eqn:4.12}
with $0<\lambda<1$, $s=1$ and $S\ge \bar S$, we get
\begin{equation*}
\Phi_{p,q}(S)-\Phi_{p,q}(1)\leq (S-\lambda)\Phi_{p,q}'(S)-(1-\lambda) \Phi_{p,q}'(1),
\end{equation*}

Observe now that \eqref{eqn:4.6} holds also for $\bar S< s < S$, because that ${\rm Crit}(u)\cap [\bar S, S]=\emptyset$. Then, the same argument in Theorem \ref{thm:4.1} to deduce that
$\Phi_{p,q}'(1)\leq0$ gives also $\Phi_{p,q}'(s)\leq0$ for any $s>\bar S$.
In particular, $\Phi_{p,q}'$ is a definitely bounded monotone function, and this
implies $\lim {\rm inf}_{S\rightarrow+\infty}\Phi_{p,q}'(S)\leq 0$. Hence, passing to the inferior limit as $S\rightarrow+\infty$ in the above inequality yields
\begin{equation*}
 \lim_{S\rightarrow+\infty} \Phi_{p,q}(S)-\Phi_{p,q}(1)\leq -(1- \lambda) \Phi_{p,q}'(1).
\end{equation*}
Letting $\lambda\rightarrow1^-$ on the right hand side leads to the second effective inequality $\lim_{S\rightarrow+\infty} \Phi_{p,q}(S)\leq \Phi_{p,q}(1)$.
\end{proof}
%%%%%%%%%%%%%%%%%%%%%%%%%%%%%%%%%%%%%%%%%%%%%%%%%%%%%%%%%%%%%%%%%%%%%%%%%%%%%%%%

\subsection{$L^p$ anisotropic Minkowski inequality}
We are ready to prove the following geometric inequalities between anisotropic $p$-capacity and total anisotropic mean curvatures.
\begin{theorem}\label{thm:5.3}
Let $\Omega\subset\mathbb R^n$ be a bounded  domain with smooth boundary. Let $(p,q)\in\Lambda$. % and $u$ is a solution of \eqref{eqn:1.1}.
Then
\begin{equation}\label{eqn:4.09}
 \left(  \int_{\partial\Omega}\big|\frac{H_F}{n-1}\big|^{q(p-1)} F(\nu)d\sigma \right)^{\frac1q} \ge \big(\frac{p-1}{n-p}\big)^{p-1} \frac{{\rm Cap}_{F,p}(\O)}{|\p \Omega|_F^{1-\frac1q}}.
\end{equation}
Moreover, equality holds in \eqref{eqn:4.09} if only if $\Omega$ is a Wulff  ball.
\end{theorem}
\begin{proof}
 Recall the formula \eqref{eqn:3.3} and also the identity \eqref{h-f}. Thus the first effective inequality $\Phi_{p, q}'(1)\le 0$
implies that
\begin{equation*}
  0\leq\int\limits_{\partial\Omega}F^{q(p-1)-1}(\nabla u)\big(H_F-{p^*}F(\nabla u)\big)F(\nu)d\sigma,
\end{equation*}
and thus, by the H\"older inequality, one gets
\begin{equation}\label{eqn:5.10}
\int\limits_{\partial\Omega}F^{q(p-1)}(\nabla u)F(\nu)d\sigma\leq\Big(\frac{n-p}{p-1}\Big)^{q(p-1)}\int\limits_{\partial\Omega}\Big|\frac {H_F}{n-1}\Big|^{q(p-1)}F(\nu)d\sigma.
\end{equation}
On the other hand, by H\"older inequality, we obtain
\begin{equation*}
\int\limits_{\partial\Omega}F^{p-1}F(\nu)d\sigma\leq \Big(\int\limits_{\partial\Omega}F^{q(p-1)}F(\nu)d\sigma\Big)^{\frac1{q}}
 \Big(\int_{\partial\Omega}F(\nu)d\sigma\Big)^{\frac{q-1}{q}}.
\end{equation*}
This combining with \eqref{eqn:2.7}, \eqref{eqn:5.10} yields \eqref{eqn:4.09}.

Assume now that equality holds in \eqref{eqn:1.02},  Then, equality holds in \eqref{eqn:4.09}, and consequently
$\Phi'_{p,q}(1)=0$. it follows that ${\rm div}X=0$ along $\p \O$. From \eqref{eqn:4.5}, we see that
$$
\frac{n-2}{n-1}\s_1^2(\kappa_F)-2\s_2(\kappa_F)=0.
$$
Then $\partial\Omega$ is anisotropic umbilical, which yields that $\partial\Omega$ is of Wullf shape.
\end{proof}

\begin{theorem}\label{thm-pq}
Let $\Omega\subset\mathbb R^n$ be a bounded domain with smooth boundary. Let $(p,q)\in\Lambda$. % and$u$ is a solution of \eqref{eqn:1.1}.
 Then
\begin{equation}\label{eqn:5.11}
\frac1{\kappa_{n-1}}\int\limits_{\partial\Omega}\Big|\frac {H_F}{n-1}\Big|^{q(p-1)}F(\nu)d\sigma\ge \Big[\big(\frac{p-1}{n-p}\big)^{p-1}\frac{1}{\kappa_{n-1}}{\rm Cap}_{p,F}(\Omega)\Big]^{1-\frac{(q-1)(p-1)}{n-p}}.
\end{equation}
Equality holds in \eqref{eqn:5.11} if only if $\Omega$ is a Wullf ball.
\end{theorem}
\begin{proof}
The second effective inequality $\Phi_{p,q}(+\infty)\leq \Phi_{p,q}(1)$, combining with the asymptotic behavior \eqref{eqn:4.2}, implies that
\begin{equation*}
\begin{aligned}
&(\frac{n-p}{p-1})^{(q-1)p^*}(\kappa_{n-1})^{\frac{(q-1)(p-1)}{n-p}}({\rm Cap}_{F, p}(\O))^{1-\frac{(q-1)(p-1)}{n-p}}\\=&\lim_{\tau\rightarrow+\infty}\Phi_{p,q}(\tau)\leq \Phi_{p,q}(1)=\int\limits_{\partial\Omega}F^{q(p-1)}(\nabla u)F(\nu)d\sigma.
\end{aligned}
\end{equation*}
Then, by using \eqref{eqn:5.10}, we obtain \eqref{eqn:5.11}. The rigidity follows from the proof as in Theorem \ref{thm:5.3}.
\end{proof}

\noindent{\it Proof of Theorems \ref{thm:1.1} and  \ref{willmore}.}

Let $q=p/(p-1)$ in \eqref{eqn:5.11}, we obtain the anisotropic $L^p$ Minkowski inequality \eqref{eqn:1.02} in Theorem \ref{thm:1.1}.
Let $q=\frac{n-1}{p-1}$ in \eqref{eqn:5.11}, we obtain the anisotropic Willmore inequality \eqref{willmore'} in Theorem \ref{willmore}.

\

%%%%%%%%%%%%%%%%%%%%%%%%%%%%%%%%%%%%%%%%%%%%%%%%%%%%%%%%%%%%%%%%%%%%%%%%%%%%%%%%%%%%%
\section{Anisotropic Minkowski Inequality for outward $F$-minimising sets}\label{sect:6}
%
%In this section, we will use the method of \cite{AFM1} to complete the proof of Theorem \ref{thm:1.2} and Corollary \ref{cor:1.3}.

\begin{definition}
Let $E\subset\mathbb R^n$ be a bounded set of finite perimeter  (s.f.p. for short). The {\it anisotropic perimeter} of $E$ is defined by
$$P_F(E)=\int_{\p^*E} F(\nu_E)d\mathcal{H}^{n-1}$$
where $\p^*E$ is the reduced boundary of $E$ and $\nu_E$ is the outward measure theoretical unit normal to $E$.
\end{definition}
\begin{remark}When $E$ has smooth boundary, then $$P_F(E)=|\p E|_F=\int_{\p E} F(\nu)d\sigma.$$
\end{remark}
\begin{definition}[Outward $F$-minimising and strictly outward $F$-minimising sets]\label{ofm}
Let $E\subset\mathbb R^n$ be a bounded set of finite perimeter. We say that $E$ is outward $F$-minimising if  \begin{eqnarray}\label{F-min}
P_F(E)\leq P_F(G)\hbox{ for }G\subset \mathbb R^n\hbox{ bounded s.f.p. with }E\subset G.
\end{eqnarray}
 We say that $E$ is strictly outward $F$-minimising if it is outward $F$-minimising and equality in \eqref{F-min} holds only if $|G\setminus E|=0$.
\end{definition}
\begin{proposition}\label{prop:6.1} Let $\O$ be an outward $F$-minimising set with smooth boundary. Then
\begin{eqnarray}\label{inf-attain}
\inf_{\substack{\O\subset U\\ \p U \hbox{smooth}}}|\p U|_F=|\p \O|_F.
\end{eqnarray}
\end{proposition}
\begin{proof}It follows directly by the definition.
\end{proof}
\begin{proposition}\label{prop:6.2}  Let $\O$ be an outward $F$-minimising set with smooth boundary. Then $H_F\ge 0$.
\end{proposition}
\begin{proof}It follows from \eqref{inf-attain} and the first variational formula of $|\p U|_F$.
\end{proof}
%Next we show the limit of ${\rm Cap}_{F, p}(\Omega)$ as $P\to 1^+$ is related to the $\inf_{\substack{\O\subset U\\ \p U \hbox{smooth}}}|\p U|_F$.
\begin{proposition}\label{thm:6.2}
Let $\Omega\subset\mathbb R^n$ be a bounded open set with smooth boundary. Then
\begin{equation}\label{limit-cap}
\lim\limits_{p\rightarrow1^+}{\rm Cap}_{F, p}(\Omega)={\rm Cap}_{F, 1}(\Omega)=\inf_{\substack{\O\subset U\\ \p U \hbox{smooth}}}|\p U|_F.
\end{equation}
%where $\Omega^*$ is the strictly outward anisotropic minimising hull of $\Omega$
\end{proposition}
\begin{proof} The second equality can be found in \cite{Maz}, Lemma 2.2.5. See also \cite{AFM}. For the convenience of the reader, we provide a proof here.

\noindent{\bf Step 1.} \begin{equation}\label{eqn:6.6}
{\rm Cap}_{F, 1}(\Omega)\ge \inf_{\substack{\O\subset U\\ \p U \hbox{smooth}}}|\p U|_F.
\end{equation}

By co-area formula, for any $u\in{C}^{\infty}_c(\mathbb R^n)$ with $u\geq\chi_{\Omega}$
we have that
\begin{equation*}
  \int\limits_{\mathbb R^n}F(\nabla u) dx\geq\int^1_0 |\p \{u>t\}|_Fdt\geq \inf_{\substack{\O\subset U\\ \p U \hbox{smooth}}}|\p U|_F.
\end{equation*}
%
%where the last equality is due to Lemma \ref{lem:6.1}.
In particular, taking the infimum over any such
$u$, we get \eqref{eqn:6.6}.

\noindent{\bf Step 2.}%
\begin{equation}\label{eqn:6.7}
{\rm Cap}_{F, 1}(\Omega)\leq\liminf_{p\rightarrow1^+}{\rm Cap}_{F, p}(\Omega).
\end{equation}

For every $u\in{C}^{\infty}_{c}(\mathbb R^n)$ with $u\geq\chi_{\Omega}$ and any $q>0$, by the definition of ${\rm Cap}_{F, 1}(\O)$ and the H\"older inequality, we have
\begin{equation}\label{eqn:6.8}
\begin{aligned}
{\rm Cap}_{F, 1}(\O)&\leq\int_{\mathbb R^n}F(\nabla u^q)dx
&\leq q\Big(\int_{\mathbb R^n}u^{\frac{(q-1)p}{p-1}}dx\Big)^{(p-1)/p}\Big(\int_{\mathbb R^n}F^p(\nabla u)dx\Big)^{1/p}.
\end{aligned}
\end{equation}
Choose %\begin{equation}\label{eqn:6.9}
$q=\frac {p(n-1)}{n-p}>1$, then $\frac{(q-1)p}{p-1}=\frac{np}{n-p}$
%\end{equation}
%
which is the Sobolev critical exponent of $p$.
The anisotropic Sobolev inequality with best constant \cite{AFTL} says that
\begin{eqnarray}\label{sobolev}
\left(\int_{\mathbb R^n}u^{\frac{np}{n-p}}dx\right)^{\frac{n-p}{np}}\le \left(\frac{\o_{n-1}}{\k_{n-1}}\right)^{\frac1n}T_{n, p} \left( \int_{\mathbb R^n}F^p(\nabla u)dx\right)^{\frac{1}{p}}.
\end{eqnarray}
where $T_{n, p}$ is the best constant in the classical Sobolev inequality, due to Talenti \cite{Talenti}.

It follows from \eqref{eqn:6.8} and \eqref{sobolev} that
\begin{equation}\label{eqn:6.10}
\begin{aligned}
{\rm Cap}_{F, 1}(\O)\leq \frac {p(n-1)}{n-p}\left(\frac{\o_{n-1}}{\k_{n-1}}\right)^{\frac{p-1}{n-p}} T_{n, p}^{\frac{n(p-1)}{n-p}}\Big(\int_{\mathbb R^n}F^p(\nabla u)
dx\Big)^{\frac {n-1}{n-p}},
\end{aligned}
\end{equation}
Taking the infimum over any $u \in{C}^{\infty}_c(\mathbb R^n)$ with
$u\geq\chi_{\Omega}$ in \eqref{eqn:6.10}, we obtain
\begin{equation}\label{eqn:6.11}
{\rm Cap}_{F, 1}(\Omega)\leq \frac {p(n-1)}{n-p}\left(\frac{\o_{n-1}}{\k_{n-1}}\right)^{\frac{p-1}{n-p}} T_{n, p}^{\frac{n(p-1)}{n-p}}
{\rm Cap}_{F, p}(\Omega)^{\frac {n-1}{n-p}}.
\end{equation}
Passing to the limit in \eqref{eqn:6.11} as $p\to 1^+$, since $\lim_{p\to 1^+}T_{n,p}= {n}^{\frac1n-1}\o_{n-1}^{-\frac1n}$, we conclude \eqref{eqn:6.7}.

%
%where $T_{n,p}$ is Talenti's best constant in the Sobolev inequality, obtained in ??. %We recall that
%the precise value of such constant is
%\begin{equation*}
%T_{n,p}=\frac1{\kappa_{n-1}^{1/2}n^{1/p}}\big(\frac{p-1}{n-p}\big)^{\frac{(p-1)}{p}}
%\Big[\frac{\Gamma(1+n/2)\Gamma(n)}{\Gamma(n/p)\Gamma(1+n-n/p)}\Big]^{1/n},??
%\end{equation*}
%
%where $\Gamma$ is Euler's Gamma function. Let $p\rightarrow1^+$, one can check that $T_{n,p}$ converges to a positive constant, precisely
%
%\begin{equation*}
%  \lim\limits_{p\rightarrow1^+}T_{n,p}=\Big(\frac1{\kappa_{n-1}}\Big)^{1/n}.
%\end{equation*}
%Let $p\rightarrow1^+$ in \eqref{eqn:6.9}, we obtain $q\rightarrow1 $.
%Then, we have
%
%\begin{equation*}
% \lim\limits_{p\rightarrow1^+}T_{n,p}^{q-1}=1.
%\end{equation*}
%
%Thus, passing to the limit in \eqref{eqn:6.11}, we get \eqref{eqn:6.7}.
%

\noindent{\bf Step 3.}
\begin{equation}\label{eqn:6.12}
\limsup_{p\rightarrow1^+}{\rm  Cap}_{F, p}(\Omega)\leq \inf_{\substack{\O\subset U\\ \p U \hbox{smooth}}}|\p U|_F.
\end{equation}
Let $U\subset \mathbb R^n$ be any open and bounded set with smooth boundary such that $\Omega\subset U$. Let $d_U$ be the distance function $d_U(x) ={\rm dist}(x, U)$. It is well-known that $d_U$ is smooth in a neighborhood of $U$.  Introduce a smooth cut-off function
$\chi_{\varepsilon}$ satisfying
\begin{equation*}
\left\{
\begin{aligned}
\chi_{\varepsilon}&=1\ \ {\rm in}\ t<\varepsilon,\\
-\frac{1}{\varepsilon}<&\chi'_{\varepsilon}(t)<0\ \ {\rm in}\ \varepsilon\leq t\leq2\varepsilon,\\
\chi_{\varepsilon}&=0\ \ {\rm in}\ t>2\varepsilon,
\end{aligned}\right.
\end{equation*}
and  set $\eta_{\varepsilon}=\chi_{\varepsilon}(d_U)$. Choosing  $\varepsilon$  small enough so that $\eta_{\varepsilon}$ is smooth. Hence
\begin{equation*}
  {\rm Cap}_{F, p}(\Omega)\leq\int\limits_{\mathbb R^n}F^p(\nabla \eta_{\varepsilon})dx
\end{equation*}
for any $p\geq 1$. Letting $p\rightarrow1^+$, by using the co-area formula and the mean value theorem, we get
\begin{eqnarray*}
\limsup_{p\rightarrow1^+}{\rm Cap}_{F, p}(\Omega)&\leq&\int\limits_{\mathbb R^n}F(\nabla \eta_{\varepsilon})dx
=\int^{2\varepsilon}_{\varepsilon}|\chi'_{\varepsilon}(t)| \int_{\{d_U=t\}}F(\nu) d\sigma dt\\&\le & \int_{\{d_U=r_\varepsilon\}}F(\nu) d\sigma,\end{eqnarray*}
%
%where in the last equality we applied the co-area formula and  $\nu_E$ is the unit Euclidean outer normal of $E$.
for some $r_{\varepsilon} \in (\varepsilon, 2\varepsilon)$.  %such that the above right hand side satisfies
%
%\begin{equation*}
%\int^{2\varepsilon}_{\varepsilon}|\chi'_{\varepsilon}(x)|dt\int_{\{d_E=t\}}F(\nu_E)d\mathcal{H}^{n-1}=
%\varepsilon|\chi'_{\varepsilon}(x)|P_F\{d_E=r_{\varepsilon}\})<P_F(\{d_E=r_{\varepsilon}\}).
%\end{equation*}
%
Let $r_{\varepsilon}\rightarrow0^+$,
we conclude that
\begin{equation*}
\limsup_{p\rightarrow1^+} {\rm Cap}_{F, p}(\O)\leq |\p U|_F
\end{equation*}
for any bounded set $U$ with smooth boundary such that $\Omega\subset U$. \eqref{eqn:6.12} follows.

%In particular, considering a
%sequence of bounded sets $\{\Omega_k\}_{k\in\mathbb N}$ with smooth boundary containing $\Omega^*$ and with
%$P_F(\Omega_k)\rightarrow P_F(\Omega^*)$ as $k\rightarrow\infty$, provided in Lemma \ref{lem:6.1}, we get \eqref{eqn:6.12}.

\noindent{\bf Step 4.} The assertion \eqref{limit-cap} follows from \eqref{eqn:6.6}, \eqref{eqn:6.7} and \eqref{eqn:6.12}.
%
%\begin{equation*}
%P_F(\Omega^*)\leq{\rm Cap_1^F}\leq\lim_{p\rightarrow1^+}{\rm inf}(\rm Cap^F_p)(\Omega)\leq\lim_{p\rightarrow1^+}{\rm sup}(\rm Cap_p)(\Omega)\leq P_F(\Omega^*),
%\end{equation*}
%completing the proof.
\end{proof}
%

%\subsection{Proof of Theorem \ref{thm:1.2} and Corollary \ref{cor:1.3}}
%

\noindent{\it Proof of Theorem \ref{thm:1.2}}. The inequality \eqref{eqn:1.03} follows from Theorem \ref{thm:1.1} and Proposition \ref{thm:6.2} by letting $p\to 1^+$.

When $\O$ is a Wulff ball, then in view of \eqref{area-vol2}, we see that
\begin{eqnarray*}
\inf_{\substack{\O\subset U\\ \p U \hbox{smooth}}}|\p U|_F=|\p \O|_F.
\end{eqnarray*}
Hence the equality in \eqref{eqn:1.03} holds.
\qed

\medskip

\noindent{\it Proof of Corollary \ref{cor:1.3}}. The first assertion follows directly from Theorem \ref{thm:1.2}, Propositions \ref{prop:6.1} and \ref{prop:6.2}.

Next we consider the equality case in \eqref{eqn:1.04} for some strictly outward  $F$-minimising set $\O$
with smooth and strictly $F$-mean-convex boundary.
Let $\{\partial\Omega_t\}_{t\in[0,T)}$ be the evolution
of $\partial\Omega$ under smooth inverse anisotropic mean curvature flow, for some $T > 0$. By \cite{DGX}, we know that the weak  inverse anisotropic mean curvature flow starting
at $\partial\Omega$ coincides with the smooth flow $\{\Omega_t\}_{t\in[0,T^*)}$ for some $0<T^*\leq T$.
In view of Proposition 3.4  of \cite{DGX} (see also Lemma 2.4 of \cite{HI}),
$\Omega_t$ is strictly outward $F$-minimising and strictly $F$-mean-convex for every  $t \in [0, T^*)$
. Then \eqref{eqn:1.04} holds for every $\partial\Omega_t$ with  $t \in [0, T^*)$.  For  $t \in [0, T^*)$, we can  define
monotonic quantity  as following
%]
\begin{equation*}
\Psi(t)=|\p \Omega_t|_F^{-\frac{n-2}{n-1}}\int_{\partial\Omega_t}H_FF(\nu)d\sigma.
\end{equation*}
%
%It is easy to see that inequality \eqref{eqn:1.04} equivalent to $\mathcal{F}(0)\geq(n- 1)\kappa_{n-1}^{1/(n-1)}$, and equality
%in \eqref{eqn:1.04} is equivalent to $\mathcal{F}(0)=(n- 1)\kappa_{n-1}^{1/(n-1)}$. By the smoothness of the flow, the
%function $\mathcal{F}(t)$ is differentiable for $t\in [0, T)$, and then a straightforward computation involving
%the standard evolution equations show that,
 by using Proposition 2.1 of \cite{Xia},
\begin{equation}\label{eqn:6.14}
\Psi'(0)=-|\p \Omega|_F^{-\frac{n-2}{n-1}}\int\limits_{\partial\Omega}\frac{1}{H_F}\left|h_F-\frac{H_F}{n-1}g_F\right|^2 d\sigma\leq 0.
\end{equation}
%where $$\dot h_F=h_F-\frac{H}{n-1}a^T_F.$$
%
%Then we have $\mathcal{F}(t)\leq\mathcal{F}(0)$.
Here $h_F-\frac{H_F}{n-1}g_F$ is the traceless anisotropic second fundamental form.
If the strict inequality $\Psi'(0)<0$ holds, it follows that
 $$\Psi(t)<\Psi(0)=(n- 1)\kappa_{n-1}^{1/(n-1)}\hbox{ for some } t\in(0, T^*), $$which contradicts \eqref{eqn:1.04} for some outward anisotropic minimising $\Omega_t$ with strictly anisotropic mean-convex boundary. Therefor $\Psi'(0)=0$. It follows from \eqref{eqn:6.14} that $\partial\Omega$ is totally anisotropic umbilical and in turn, $\partial\Omega$ is of a Wulff shape.
\qed

\medskip

\noindent{\it Proof of Corollary \ref{cor:1.4}}. For any $U\subset \rr^n$ bounded with smooth boundary and $\O\subset U$, by the Wulff inequality \eqref{wulff-ineq}, we have
\begin{eqnarray}\label{area-vol}
\frac{|\p U|_F}{\k_{n-1}}\ge\left(\frac{n|U|}{\k_{n-1}}\right)^{\frac{n-1}{n}}\ge \left(\frac{n|\O|}{\k_{n-1}}\right)^{\frac{n-1}{n}}.
\end{eqnarray}
Thus
\begin{eqnarray}\label{area-vol2}
\frac{1}{\k_{n-1}}\inf_{\substack{\O\subset U\\ \p U \hbox{smooth}}}|\p U|_F\ge \left(\frac{n|\O|}{\k_{n-1}}\right)^{\frac{n-1}{n}}.
\end{eqnarray}
The inequality follows from Theorem \ref{thm:1.2} and \eqref{area-vol2}. The equality classification for \eqref{eqn:1.05} follows from that of the Wulff inequality.
\qed

\

\noindent{\bf Acknowledgements.} This work was supported by  NSFC (Grant No. 11871406). We would like to thank Dr. M. Fogagnolo for attracting our attention to their recent preprint \cite{FM} and explaining the new reformation of strictly outward minimizing hull to us. We would also like to thank Professors  V. Agostiniani, L. Mazzieri and Deping Ye for their interest.

\

%%%%%%%%%%%%%%%%%%%%%%%%%%%%%%%%%%%%%%%%%%%%%%%%%%%%%%%%%%%%%%%%%%%%%%%%%%%%

\end{document}